\documentclass[11pt,reqno]{amsart}
\usepackage[T1]{fontenc}
\usepackage[letterpaper,margin=1.25in]{geometry}  
\usepackage{verbatim}
\usepackage{amssymb,amsmath,amsthm}
\usepackage{booktabs}
\usepackage{float}
\usepackage[colorlinks=true,linkcolor=blue,citecolor=blue,urlcolor=blue]{hyperref}
\usepackage{cleveref}
\DeclareMathOperator{\rk}{rk}
\DeclareMathOperator{\Aut}{Aut}
\DeclareMathOperator{\End}{End}
\DeclareMathOperator{\Gal}{Gal}
\DeclareMathOperator{\Cl}{Cl}
\DeclareMathOperator{\gon}{gon}
\DeclareMathOperator{\ddiv}{div}

\DeclareMathOperator{\Jac}{Jac}

\DeclareMathOperator{\PSL}{PSL}
\DeclareMathOperator{\SL}{SL}
\DeclareMathOperator{\ord}{ord}
\DeclareMathOperator{\Frob}{Frob}
\DeclareMathOperator{\Ver}{Ver}
\DeclareMathOperator{\red}{red}
\DeclareMathOperator{\disc}{disc}

\DeclareMathOperator{\Norm}{N}

\newcommand{\Q}{\mathbb{Q}}
\newcommand{\Z}{\mathbb{Z}}
\newcommand{\C}{\mathbb{C}}
\newcommand{\F}{\mathbb{F}}

\newcommand{\Qbar}{\overline{\Q}}
\newcommand{\tors}{\mathrm{tors}}
\newcommand{\torz}[1]{\Z/#1\Z}
\newcommand{\tg}[2]{\Z/#1\Z\times\Z/#2\Z}
\newcommand{\dm}[1]{\langle #1\rangle}
\newcommand{\ClCusp}{\Cl^{c}}

\newtheorem{theorem}{Theorem}[section]
\newtheorem*{theorem*}{Theorem}
\newtheorem{lemma}[theorem]{Lemma}
\newtheorem{proposition}[theorem]{Proposition}
\newtheorem{corollary}[theorem]{Corollary}
\theoremstyle{definition}
\newtheorem{remark}[theorem]{Remark}

\newcommand{\lmfdbec}[3]{\href{https://www.lmfdb.org/EllipticCurve/Q/#1/#2/#3}{#1.#2#3}}

\newcommand{\lmfdcharorbit}[2]{\href{https://www.lmfdb.org/Character/Dirichlet/#1/#2}{#1.#2}}
\newcommand{\lmfdbnewform}[4]{\href{https://www.lmfdb.org/ModularForm/GL2/Q/holomorphic/#1/#2/#3/#4}{#1.#2.#3.#4}}
\newcommand{\githubbare}[1]{\href{https://github.com/F-Najman/quintic-torsion/blob/main/#1}{\path{#1}}}
\newcommand{\lmfdbnf}[1]{\href{https://www.lmfdb.org/NumberField/#1}{#1}}

\newcommand{\gitlink}[2]{\href{https://github.com/F-Najman/quintic-torsion/blob/main/#1}{#2}}

\title[Torsion of elliptic curves over quintic fields]
      {Classification of torsion of elliptic curves\\ over quintic fields}
\author{Filip Najman}
\address{University of Zagreb Faculty of Science, Department of Mathematics,
Bijeni\v cka cesta 30, 10000 Zagreb, Croatia}
\email{fnajman@math.hr}
\thanks{The author was supported by the Croatian Science Foundation under the project no.\ IP-2022-10-5008, by the project ``Implementation of cutting-edge research and its application as part of the Scientific Center of Excellence for Quantum and Complex Systems, and Representations of Lie Algebras'', PK.1.1.10.0004, co-financed by the European Union through the European Regional Development Fund -- Competitiveness and Cohesion Programme 2021--2027, and by the European Union -- NextGenerationEU through the National Recovery and Resilience Plan 2021--2026, via an institutional grant of the University of Zagreb Faculty of Science, IK IA 1.1.3, Impact4Math.}
\dedicatory{Dedicated with gratitude and admiration to my advisor Andrej Dujella on the
occasion of his 60th birthday.}

\begin{document}

\begin{abstract}
We determine all the groups that appear as the torsion group of an elliptic curve over a quintic number field. Apart from the groups that already occur infinitely often, which were determined by Derickx and Sutherland, exactly three groups occur: $\torz{28}$, $\torz{30}$ and $\tg{2}{18}$. Up to isomorphism of the pair $(K,E)$, the first and the third are each realized by a single elliptic curve and the second by two curves, which are $2$-isogenous over a common quintic field. The curves realizing $\torz{28}$ and $\torz{30}$ were found by van Hoeij, while the group $\tg{2}{18}$ is new, and $5$ is the smallest degree in which a non-cyclic sporadic torsion group occurs. The methods improve on those developed by Derickx and Najman, and are based on Hecke sieves and the arithmetic of cuspidal divisor classes.
\end{abstract}
\date{\today}
\maketitle

\section{Introduction}

Let $E$ be an elliptic curve over a number field $K$. By the Mordell--Weil theorem
$E(K)\simeq E(K)_\tors\times \Z^r$, and by Merel's theorem \cite{merel} the torsion subgroup
$E(K)_\tors$ is bounded in terms of $d=[K:\Q]$ alone. Hence, for every $d\geq 1$ the set
\[
\Phi(d):=\{\,E(K)_\tors \;:\; [K:\Q]=d,\ E/K \text{ an elliptic curve}\,\}
\]
is finite. Determining $\Phi(d)$ has been one of the central problems in the arithmetic of
elliptic curves. The set $\Phi(1)$ was determined by Mazur \cite{mazur}, building on work of
Kubert \cite{kubert}, and $\Phi(2)$ by Kamienny \cite{kamienny}, building on work of Kenku and
Momose \cite{KM}. The set $\Phi(3)$ was determined by Derickx, Etropolski, van Hoeij, Morrow
and Zureick-Brown \cite{DEHMZ}, building on work of Jeon, Kim and Schweizer \cite{JKS}, Parent
\cite{parent} and Najman \cite{najman}, and $\Phi(4)$ by Derickx and Najman \cite{DN},
building on work of Jeon, Kim and Park \cite{JKP} and Bruin and Najman \cite{bruin_najman2016}. In this paper we determine $\Phi(5)$.

Write $\Phi^\infty(d)\subseteq\Phi(d)$ for the subset of groups that occur as $E(K)_\tors$ for
infinitely many $\overline\Q$-isomorphism classes of elliptic curves $E$ over number fields $K$
of degree $d$. A group in $\Phi(d)\setminus\Phi^\infty(d)$ will be called a
\emph{sporadic torsion group in degree $d$}. Following \cite[\S 1]{DN}, a point of
degree $d$ on a curve $X$ is called \emph{sporadic} if $X$ has only finitely many points of
degree at most $d$. Note that this is a property of a point on a modular curve, not of a group.
We have $\Phi^\infty(d)=\Phi(d)$ for $d=1,2$. There is exactly one sporadic torsion group
in degree $3$, namely $\torz{21}$, which is realized by a single elliptic curve, the base
change to $\Q(\zeta_9)^+$ of the curve \lmfdbec{162}{c}{3} \cite{najman,DEHMZ}, and there
are no sporadic torsion groups in degree $4$ \cite{DN}. The set $\Phi^\infty(5)$ was determined by Derickx and
Sutherland \cite[Theorem 1.1]{DS}, the cyclic part being due to Derickx and van Hoeij
\cite[Theorem 3]{DvH}:
\begin{equation}\label{eq:phiinf5}
\Phi^\infty(5)=\{\torz{n}\;:\;1\leq n\leq 25,\ n\neq 23\}\ \cup\
                \{\tg{2}{2n}\;:\;1\leq n\leq 8\}.
\end{equation}
Our main result is the following.

\begin{theorem}\label{thm:main}
Let $K$ be a number field of degree $5$ and let $E$ be an elliptic curve over $K$. Then
$E(K)_\tors$ is isomorphic to one of the following $35$ groups:
\begin{alignat*}{2}
&\torz{n}, \qquad && n=1,\dots,22,\ 24,\ 25,\ 28,\ 30,\\
&\tg{2}{2n}, \qquad && n=1,\dots, 9.
\end{alignat*}
Equivalently, $\Phi(5)=\Phi^\infty(5)\cup\{\torz{28},\torz{30},\tg{2}{18}\}$. Each of the
groups in $\Phi^\infty(5)$ is the torsion group of infinitely many pairs $(K,E)$, whereas for
the three exceptional groups there are $1$, $2$ and $1$ isomorphism classes of pairs $(K,E)$,
respectively. The corresponding curves have quintic $j$-invariants and hence give
$5$, $10$ and $5$ $\overline\Q$-isomorphism classes. They are listed explicitly
in \Cref{thm:sporadic}.
\end{theorem}

Thus $d=5$ is the first degree in which more than one sporadic torsion group occurs, and
$\tg{2}{18}$ is the first non-cyclic sporadic torsion group in any of the degrees $d\leq5$ for
which $\Phi(d)$ is known. The three sporadic
groups are described completely by the following theorem, which is proved in \Cref{sec:sporadic}.

\begin{theorem}\label{thm:sporadic}
Let $K$ be a number field of degree $5$ and let $E/K$ be an elliptic curve such that
$E(K)_\tors$ is not in $\Phi^\infty(5)$. Then, up to isomorphism of the pair $(K,E)$, exactly
one of the following holds.
\begin{enumerate}
\item $E(K)_\tors\simeq \torz{28}$, $K=\Q(\alpha)$ with
      $\alpha^5-\alpha^4-2\alpha^3-\alpha^2+2\alpha+2=0$, and $E$ is the curve
      \eqref{eq:E28} in Section \ref{sec:sporadic}.
\item $E(K)_\tors\simeq \torz{30}$, $K=\Q(\alpha)$ with
      $\alpha^5+\alpha^4-3\alpha^3+3\alpha+1=0$, and $E$ is one of the two curves
      \eqref{eq:E30a}, \eqref{eq:E30b} in Section \ref{sec:sporadic}, which are $2$-isogenous over $K$.
\item $E(K)_\tors\simeq \tg{2}{18}$, $K=\Q(w)$ with $w^5-2w^4+2w^3-3w^2-w+4=0$, and $E$ is the
      curve \eqref{eq:E218} in Section \ref{sec:sporadic}.
\end{enumerate}
In all three cases $K$ has signature $(3,1)$ and class number $1$ and the Galois group of its
normal closure is $S_5$. Each of the four curves is without complex multiplication and has a
$j$-invariant of degree $5$ over $\Q$. The discriminants, LMFDB labels and conductors are
displayed in \Cref{table:fields}.
\end{theorem}

The curves in cases (1) and (2) were found by van Hoeij \cite{vanHoeij}. What is new here is that his lists of quintic points on $X_1(28)$ and $X_1(30)$ are complete, and that the corresponding torsion groups are $\torz{28}$ and $\torz{30}$. The curve in case (3) is new. In the course of the proof we determine all
quintic points on the three relevant modular curves.

\begin{theorem}\label{thm:points}
\begin{enumerate}
\item $X_1(28)$ has exactly $6$ non-cuspidal closed points of degree $5$; they form a single
      free orbit under the group $\Delta(28)$ of diamond operators, and correspond to the six
      pairs $\{P,-P\}$ of points $P$ of order $28$ on the curve $E$ of
      \Cref{thm:sporadic}(1).
\item $X_1(30)$ has exactly $8$ non-cuspidal closed points of degree $5$; they form two free
      orbits under $\Delta(30)$, and correspond to the four pairs $\{P,-P\}$ of points $P$ of
      order $30$ on each of the two curves of \Cref{thm:sporadic}(2).
\item $X_1(2,18)$ has exactly $18$ non-cuspidal closed points of degree $5$; they correspond
      to the $18$ equivalence classes, modulo $\pm1$, of embeddings
      $\tg{2}{18}\hookrightarrow E(K)$, for the curve $E$ of \Cref{thm:sporadic}(3).
\end{enumerate}
The residue field of each of these points is the corresponding quintic
field $K$ of \Cref{thm:sporadic}.
\end{theorem}

\Cref{thm:main} also settles, for every level at once, the question of which modular curves of
the two families $X_1(n)$ and $X_1(2,2n)$ carry a point of degree $5$.

\begin{corollary}\label{cor:degree5points}
\begin{enumerate}
\item $X_1(n)$ has a non-cuspidal point of degree $5$ if and only if
$n\in\{1,\dots,22,24,25,28,30\}$. For $n\leq 25$, $n\neq23$, there are infinitely many such
points, while for $n=28$ there are exactly $6$ and for $n=30$ exactly $8$.
\item $X_1(2,2n)$ has a non-cuspidal point of degree $5$ if and only if $n\leq 9$. For
$n\leq 8$ there are infinitely many such points, and for $n=9$ there are exactly $18$.
\item In particular, $X_1(28)$ and $X_1(30)$ are the only curves in these two families that have a \emph{non-cuspidal} sporadic point of degree $5$.
\end{enumerate}
\end{corollary}

The proof of \Cref{cor:degree5points} is given in \Cref{sec:sporadic}, together with the reason
why the word \emph{non-cuspidal} cannot be dropped from part~(3): the curve $X_1(44)$ has a
cuspidal sporadic point of degree $5$. The three exceptional groups behave differently as far as
the modular curves are concerned. For $\torz{28}$ and $\torz{30}$ the corresponding points are
sporadic points of $X_1(28)$ and $X_1(30)$, while for $\tg{2}{18}$ they are not, because
$\gon_\Q X_1(2,18)=4$ and this curve has infinitely many quartic points. They are nevertheless
\emph{isolated} in the sense of Bourdon, Ejder, Liu, Odumodu and Viray
\cite[Definition 4.1]{BELOV}. The example has a feature not seen in lower degrees: by
\cite[Theorem 3.6]{JKP} we have $\tg{2}{18}\in\Phi^\infty(4)$, so $\tg{2}{18}$ is the torsion
group of infinitely many elliptic curves over quartic fields, and over quintic fields of a
unique pair $(K,E)$ up to isomorphism. See \Cref{sec:sporadic}.

\subsection*{Outline of the proof and of the paper}
As in \cite{DN} and previous works, the classification reduces to showing that certain modular curves
$X_1(m,n)$ have no non-cuspidal points of degree $5$. In \Cref{sec:groups} we recall why
$45$ groups have to be analyzed. Nineteen of them were eliminated, for quintic fields,
in \cite[\S 7]{DN}, and four more, namely $\torz{4n}$ for $9\leq n\leq12$, were reduced there
to the four groups $\tg{2}{2n}$. We recall that argument in \Cref{sec:literature}, with the numerical
conditions checked for $d=5$ rather than $d=4$. One point of formulation, which is what puts $\torz{77}$ and $\torz{85}$ outside its reach in degree $5$, is discussed in \Cref{rem:erratum}.

The twenty-two groups that \cite{DN} does not reach are treated here. The group $\torz{36}$ then
needs one further argument, because $\tg{2}{18}$ turns out to occur. The main tools are:

\begin{itemize}
\item the criterion of \cite[Proposition 5.2]{DN} for curves whose Jacobian has rank $0$
      (\Cref{sec:rank0});
\item the Hecke sieve of \cite[Proposition 5.5]{DN}, together with
      \cite[Lemma 5.10]{DN}, which controls the action of $A_q$ on cuspidal divisors over
      $\F_p$ of arbitrary degree and which we restate as \Cref{lem:CC} in the form needed for our
      computations and for the twisted operators of \Cref{sec:posrank} (\Cref{sec:sieve});
\item a twisted Hecke sieve for the positive analytic rank cases $n=57,63,65$, in which $A_q$ is
       replaced by $A_q(\dm{a}-1)$ or $A_q\pi^*\pi_*$ for a subgroup $H$, together with a $p$-adic diamond argument
      at an odd prime that handles the totally cuspidal reduction (\Cref{sec:posrank});
\item a direct analysis of the cuspidal class group, which determines \emph{all} non-cuspidal
      closed points of degree $5$ on $X_1(2,18)$, $X_1(2,20)$, $X_1(2,24)$, $X_1(28)$ and
      $X_1(30)$, and, on all of these but $X_1(2,18)$, the complete set of rational effective
      divisors of degree $5$ (\Cref{sec:direct}).
\end{itemize}

The computations for $n=57,63,65$ take place on modular curves of genus $85$, $97$ and $121$
and would not have been feasible with the Hecke operator routines of \cite{DN} and \cite{mdmagma}. In
\Cref{sec:fast} we describe and prove correct two routines that made them possible: one
computes $T_q$ on a closed point of $X_1(n)_{\F_p}$ by working, for each Frobenius orbit of
subgroups $G\subseteq E[q]$, in the field of definition of $G$ rather than in the splitting
field of the whole $q$-division polynomial, and the other computes orbit representatives for the
diamond action by isomorphism tests of pairs $(E,aP)$ instead of by function-field
computations. On $X_1(45)_{\F_7}$ the first of these reduces the computation of $T_{11}$ on a
degree $5$ place from $7114$ to $4.3$ seconds.

\Cref{sec:sporadic} assembles the proofs of \Cref{thm:main,thm:sporadic,thm:points}, and
\Cref{sec:computations} describes the computations.

\subsection*{Data availability and reproducibility}
All the code and all the logs of the computations described in this paper are available at
\begin{center}\url{https://github.com/F-Najman/quintic-torsion}\end{center}
The \texttt{README} there lists, statement by statement, which script establishes each claim and which log it produced. In the digital form of this article, the words reporting a computation, such as ``we \gitlink{direct_analysis.m}{compute}'', are hyperlinks to the script that carries it out.

\subsection*{AI use disclosure} 
AI models were used in the work leading to this paper. Anthropic's Claude Fable 5.1 and Opus 5 were used for mathematical interactions and exploration, to write, audit and run the Magma and Python code and to organize the resulting data, and in the preparation of the manuscript. In particular, Fable 5.1 suggested Lemmas 10.1 and 10.2 and their proofs. Drafts of the manuscript and its computations were reviewed by OpenAI's GPT-6 Astra and by Anthropic's Claude Opus 5.5, which suggested several corrections and expositional improvements. The author takes full responsibility for the content of the paper.

\subsection*{Acknowledgements}
I thank Pete Clark and Maarten Derickx for helpful comments.

\section{Notation and auxiliary results}\label{sec:prelim}

Throughout, $m\mid n$ are positive integers with $mn>4$, and $X:=X_1(m,n)$ denotes the modular
curve parametrizing triples $(E,P,Q)$ consisting of a generalized elliptic curve $E$ and points
$P,Q$ generating a subgroup of $E[n]$ isomorphic to $\tg{m}{n}$. It is a smooth projective
geometrically irreducible curve over $\Z[\zeta_m,\tfrac1n]$, and it is a fine moduli space when
$n>4$. Every curve occurring below has $n\geq18$, so we use the fine moduli interpretation throughout. We write $Y_1(m,n)$ for the open subscheme obtained by removing the cusps, and
$J:=J_1(m,n)$ for the Jacobian of $X$. We abbreviate $X_1(n):=X_1(1,n)$ and
$J_1(n):=J_1(1,n)$. In this paper only $m=1$ and $m=2$ occur, so $X$ is always a curve
over $\Q$. For $d\geq1$ we write $X^{(d)}$ for the $d$-th symmetric power of $X$, so that
$X^{(d)}(\Q)$ is the set of effective $\Q$-rational divisors of degree $d$ on $X$. We write
$\ClCusp_\Q X$ for the subgroup of $J(\Q)$ generated by the classes of the $\Q$-rational
divisors of degree $0$ supported on the cusps.

The cusps of $X_1(m,n)$ correspond to N\'eron $k$-gons with $m\mid k\mid n$ together with a level
structure, i.e. an injective homomorphism $\varphi:\tg{m}{n}\to \mathbb G_m\times\torz{k}$
which is surjective on the second factor (see \cite[\S 2]{DN}). In particular all cusps are
defined over $\Q(\zeta_n)$, and $\Gal(\Q(\zeta_n)/\Q)=(\torz n)^\times$ acts on them through
its action on $\mathbb G_m$. We denote by $c_0\in X_1(m,n)(\Z[\tfrac1n])$ the cusp given by the
N\'eron $n$-gon together with the level structure $(a,b)\mapsto (\zeta_m^a,b)$ (see \cite[Lemma 2.8]{DN}), which is
$\Q$-rational for $m\leq 2$.

For a congruence subgroup $\Gamma\subseteq\SL_2(\Z)$ we write $\overline\Gamma$ for its image
in $\PSL_2(\Z)$. For $n>2$ we have
$[\PSL_2(\Z):\overline{\Gamma_1(n)}]=\tfrac12[\SL_2(\Z):\Gamma_1(n)]$.

For $a\in(\torz n)^\times$ we write $\dm a$ for the diamond automorphism
$(E,P,Q)\mapsto (E,aP,aQ)$. It depends only on $a$ modulo $\pm 1$, and we call
$\Delta(n):=(\torz n)^\times/\{\pm1\}$ the group of diamond operators. We write
$\Z[\Delta(n)]$ for its group ring. It acts on divisors, on divisor classes and on $J$, and it
commutes with $T_q$ and with $\dm q$, hence with $A_q$. For a prime $q\nmid 2n$
we write $T_q$ for the Hecke correspondence
$T_q(E,P,Q)=\sum_{G}(E/G,P\bmod G,Q\bmod G)$, the sum being over the $q+1$ subgroups
$G\subseteq E[q]$ of order $q$, and we set
\[
A_q:=T_q-q\dm q-1 .
\]
We will use the following facts from \cite[\S\S 2 and 5]{DN}.

\begin{proposition}\label{prop:tools}
Let $m\in\{1,2\}$, $m\mid n$, and let $q\nmid 2n$ be a prime.
\begin{enumerate}
\item \textup{(Eichler--Shimura)} $T_{q,\F_q}=\Frob_q+\dm q_{\F_q,*}\Ver_q$ on $J_{\F_q}$.
\item If $q$ splits completely in a number field $L$ then $A_q\bigl(J(L)_\tors\bigr)=0$.
\item $A_q(c_0)=0$ as a divisor.
\item If $\rk J(\Q)=0$ then $[A_q(D)]=0$ for every $D\in X^{(d)}(\Q)$ and every $d$.
\end{enumerate}
\end{proposition}

\begin{proof}
These are, respectively, \cite[(2.2)]{DN}, \cite[Proposition 2.6]{DN},
\cite[Lemma 2.8]{DN} and \cite[Lemma 5.4]{DN}.
\end{proof}

We will also use the following results.

\begin{itemize}
\item (\emph{Manin--Drinfeld \cite{manin,drinfeld}}) the class of a degree $0$ cuspidal divisor
      on $X$ is torsion in $J(\Qbar)$;
\item (\emph{Katz \cite[Appendix]{katz}}) if $p>2$ is a prime of good reduction of $J$ and
      $\wp$ is a prime of a number field $L$ above $p$ which is unramified in $L$, then
      reduction modulo $\wp$ is injective on $J(L)_\tors$;
\item (\emph{Parent \cite[Proposition 2.4]{parent}}) if $2$ is unramified in $L$ and
      $2\nmid n$, then the kernel of reduction modulo a prime of $L$ above $2$ on
      $J(L)_\tors$ is trivial or of exponent $2$;
\item (\emph{Abramovich \cite[Theorem 0.1]{abramovich}, with the bound
      $\lambda_1\geq 975/4096$ of Kim--Sarnak \cite[Appendix 2]{kim}}) for a congruence subgroup
      $\Gamma\subseteq\PSL_2(\Z)$,
      \[
        \gon_\C X_\Gamma\geq\tfrac{325}{2^{15}}\,[\PSL_2(\Z):\Gamma];
      \]
\item (\emph{Frey \cite{frey}, Derickx--Sutherland
      \cite[Propositions 2.2 and 2.3, Corollary 2.4]{DS}}) write $d(X)$ for the least integer
      $d$ such that $X$ has infinitely many points of degree $d$. Then
      $d(X)\leq\gon_\Q X\leq 2\,d(X)$, and if $\rk J(\Q)=0$ then $d(X)=\gon_\Q X$. In that
      case $X$ has infinitely many points of degree $d$ if and only if $\Q(X)$ contains a
      function of degree exactly $d$. Note that this is stronger than $\gon_\Q X\leq d$, as
      $X_1(2,18)$ will show;
\item (\emph{Clark--Corn--Rice--Stankewicz \cite[Sections 4.1--4.5]{CCRS14}}) the torsion
      subgroup of a CM elliptic curve over a number field of degree at most $5$ has exponent
      at most $21$, so the least degree $d_{\mathrm{CM}}X_1(n)$ of a CM point on $X_1(n)$
      satisfies $d_{\mathrm{CM}}X_1(n)>5$ for every $n\geq22$. The exact values of
      $d_{\mathrm{CM}}X_1(n)$ that we quote are taken from the repository accompanying
      Clark--Genao--Pollack--Saia \cite{CGPS}.
\end{itemize}

Finally, we state the two standard reduction lemmas.

\begin{lemma}\label{lem:additive}
Let $mn>4$, let $K$ be a number field, let $E/K$ be an elliptic curve with
$\tg{m}{n}\subseteq E(K)$, and let $\wp$ be a prime of $K$ above a rational prime $p\nmid n$.
Then $E$ does not have additive reduction at $\wp$.
\end{lemma}

\begin{proof}
This is contained in the proof of \cite[Proposition 5.2]{DN}. If $E$ had additive reduction at
$\wp$, the special fibre of the N\'eron model would have $a\cdot\#\F_\wp$ points over $\F_\wp$,
where $a\leq4$ is the order of the component group and $\#\F_\wp$ is a power of $p$. Reduction is
injective on the prime-to-$p$ torsion of $E(K)$, so $\tg{m}{n}$ embeds into that group and
$mn\mid a\cdot\#\F_\wp$. As $m\mid n$ and $p\nmid n$ we have $\gcd(mn,p)=1$, so $mn\mid a\leq4$,
contradicting $mn>4$.
\end{proof}

\begin{lemma}[{\cite[Lemma 5.3]{DN}}]\label{lem:badred}
Let $p\nmid n$ be a prime and suppose that for every $d'\leq d$ there is no elliptic curve
$E'/\F_{p^{d'}}$ with $\tg{m}{n}\subseteq E'(\F_{p^{d'}})$. Then no elliptic curve $E$ over a
number field $K$ of degree $d$ with $\tg{m}{n}\subseteq E(K)$ has good reduction at any prime
of $K$ above $p$.
\end{lemma}

Whenever we invoke \Cref{lem:badred}, we have \gitlink{torsion_over_Fq.m}{checked} the hypothesis by enumerating all elliptic curves over $\F_{p^{d'}}$, i.e. all $j$-invariants and all twists, in Magma (see \Cref{sec:computations}). In most cases this is also visible from Waterhouse's description
\cite[Theorem 4.1]{waterhouse} of the possible traces of Frobenius, as in
\cite[Corollary 4.6]{DN}.

\section{The groups that have to be considered}\label{sec:groups}

Call a finite abelian group $G$ a \emph{quintic torsion group} if $G\simeq E(K)_\tors$ for some
elliptic curve $E$ over some number field $K$ of degree $5$. Every such $G$ is generated by at
most two elements, so $G\simeq\tg{m}{n}$ with $m\mid n$.

\begin{lemma}\label{lem:45}
Let $\mathcal{G}$ be the set of groups $G=\tg{m}{n}$, $m\mid n$, such that
\begin{enumerate}
\item $G\notin\Phi^\infty(5)$, while $H\in\Phi^\infty(5)$ for every proper subgroup $H\subsetneq G$;
\item $\varphi(m)\mid 5$;
\item every prime divisor of $mn$ is at most $19$.
\end{enumerate}
Then $\mathcal{G}$ consists of the following $45$ groups:
\[
\begin{array}{l}
\torz{n}\ \text{ for }\ n\in\{26,27,28,30,32,33,34,35,36,38,39,40,42,44,45,48,\\
\qquad\quad 49,50,51,55,57,63,65,75,77,85,91,95,119,121,125,133,\\
\qquad\quad 143,169,187,209,221,247,289,323,361\},\\[2pt]
\tg{2}{2n}\ \text{ for }\ n\in\{9,10,11,12\}.
\end{array}
\]
Moreover every group in $\Phi(5)\setminus\Phi^\infty(5)$ contains a member of $\mathcal G$.
\end{lemma}

\begin{proof}
Conditions (2) and (3) are necessary for $G$ to be a quintic torsion group. The Weil pairing
forces $\Q(\zeta_m)\subseteq K$, so $\varphi(m)\mid 5$ and $m\in\{1,2\}$. If $p\mid mn$, then there is a point of order $p$ over a quintic field, so $p\in S(5)$, the set of
primes $\leq 19$ by \cite[Theorem 1.2]{DKSS}. The set $\Phi^\infty(5)$ of \eqref{eq:phiinf5} is
closed under passing to subgroups, so every $G\notin\Phi^\infty(5)$ satisfying (2) and (3)
contains a member of $\mathcal G$, which is the last assertion. We \gitlink{group_list.py}{compute} the explicit list by a finite enumeration, and it agrees with the list in \cite[\S 7]{DN}.
\end{proof}

Let $G=\tg{m}{n}\in\mathcal G$. There exists an elliptic curve $E$ over a quintic field $K$
with $G\hookrightarrow E(K)$ if and only if $X_1(m,n)$ has a non-cuspidal closed point of degree
$5$. 
Such a pair $(E,K)$ gives a non-cuspidal point of $X_1(m,n)$ whose residue field
lies in $K$, hence has degree $1$ or $5$, and degree $1$ is excluded by
Mazur's theorem \cite{mazur} since $G\notin\Phi(1)$. Conversely a non-cuspidal closed point of
degree $5$ is represented by such a pair, because $X_1(m,n)$ is a fine moduli space. This is a
statement about $G$ as a subgroup, so it remains to determine these points and the full torsion
groups of the elliptic curves they represent. The latter is carried out in
\Cref{sec:sporadic}.

Three of the members of $\mathcal G$ need not be treated separately.

\begin{lemma}\label{lem:4n}
For $9\leq n\leq 12$ there is a $\Q$-rational morphism
$X_1(4n)\to X_1(2,2n)$, $(E,P)\mapsto \bigl(E/\dm{2nP},\,Q,\,P\bmod \dm{2nP}\bigr)$, where $Q$
generates $E[2]/\dm{2nP}$. Hence, if $X_1(2,2n)$ has no non-cuspidal point of degree
dividing $5$, then neither does $X_1(4n)$, and $\torz{4n}$ is not a quintic torsion group.
\end{lemma}

\begin{proof}
The morphism is the one described in \cite[\S 2]{DN} (see also \cite[\S 2]{DS}). It maps
non-cuspidal points to non-cuspidal points and does not increase the degree of a point.
\end{proof}

For $n=10,11,12$ we will show in \Cref{sec:sieve,sec:direct} that $X_1(2,2n)$ has no
non-cuspidal quintic points, so \Cref{lem:4n} disposes of $\torz{40},\torz{44},\torz{48}$. For
$n=9$ the hypothesis of \Cref{lem:4n} fails, and $\torz{36}$ requires a separate argument,
given in \Cref{prop:36}.

The rest of the paper treats the $45$ groups of \Cref{lem:45} one family at a time;
\Cref{table:summary} shows where each of them is dealt with.

\begin{table}[H]
\centering
\small
\footnotesize
\begin{tabular}{@{}l@{\quad}l@{}}
\toprule
$G\in\mathcal G$ & treated in\\
\midrule
$\torz n$, $n=49,51,55,75$ & \Cref{prop:49_51_55_75} (from \cite[Proposition 7.3]{DN})\\
$\torz n$, $n=91,95,119,121,125,133,143,$ & \Cref{prop:from_DN}, from \cite[Rem.~6.2, Thm.~7.1]{DN}\\
\quad $169,187,209,221,247,289,323,361$ & \quad (see \Cref{rem:erratum})\\
$\torz n$, $n=27,35,39,50$ & \Cref{prop:pure}\\
$\torz n$, $n=26,32,33,34,38,42,45$; $\tg{2}{22}$ & \Cref{prop:sieve} (Hecke sieve)\\
$\torz n$, $n=77,85$ & \Cref{prop:77_85} (diamond argument at $2$)\\
$\torz n$, $n=57,63,65$ & \Cref{prop:57_63_65} (twisted Hecke sieve)\\
$\tg{2}{20}$, $\tg{2}{24}$ & \Cref{prop:2_20_2_24} (direct analysis)\\
$\torz n$, $n=40,44,48$ & \Cref{lem:4n}\\
$\torz{36}$ & \Cref{prop:36}\\
\midrule
$\torz{28}$, $\torz{30}$, $\tg{2}{18}$ & \emph{are} quintic torsion groups; \Cref{prop:28_30_218}\\
\bottomrule
\end{tabular}
\medskip
\caption{Treatment of the $45$ groups of \Cref{lem:45}.}
\label{table:summary}
\end{table}

\section{The groups eliminated in \texorpdfstring{\cite{DN}}{[DN]}}
\label{sec:literature}

Nineteen of the $45$ groups of \Cref{lem:45} were eliminated over quintic fields in
\cite[\S 7]{DN}, and three more follow from \Cref{lem:4n} once $\tg{2}{2n}$, $n=10,11,12$, is
disposed of. We recall the arguments, because one
of the hypotheses used there needs to be corrected, and because the numerical conditions have
to be checked with $d=5$ rather than $d=4$.

\subsection{The global method}
The main tool is \cite[Proposition 4.3]{DN}, due to Derickx: let $n$ be an integer, $q$ an odd
prime with $q\nmid n$, $H\subseteq(\torz n)^\times$ a subgroup containing $-1$, and write
$X_H:=X_1(n)/H$ and $J_H:=\Jac(X_H)$. Let $a\in(\torz n)^\times/H$ be such that
$(\dm a-1)J_H(\Q)$ is finite. Put
\[
k_q=\begin{cases} q+1 & \text{if $J_H(\Q)$ is finite},\\
                  2q+1 & \text{if $a\in\{q,q^{-1}\}$},\\
                  2(q+1) & \text{otherwise.}\end{cases}
\]
If $k_qd<\gon_\Q X_H$ then every $D\in X_H^{(d)}(\Q)$ whose support contains no cusp is a sum of
orbits under $\dm a$. The cited proposition also covers $q=2$, but only under an additional
hypothesis on the $2$-primary part of $(\dm a-1)J_H(\Q)$, that it be trivial or killed by $A_2$. We never use that case, and we come back to it in \Cref{sec:posrank}. Taking $H=\{\pm1\}$ and using Abramovich's bound, the inequality
$k_qd<\gon_\Q X_1(n)$ holds as soon as $d<b(n)$, where
\begin{equation}\label{eq:bn}
b(n):=\frac{325}{2^{15}}\cdot\frac{[\PSL_2(\Z):\overline{\Gamma_1(n)}]}{k_q}.
\end{equation}
The diamond operators act freely on non-CM points, so if in addition $\ord\dm a\nmid d$, then a
$\dm a$-invariant effective divisor of degree $d$ must contain a CM point in its support, which is impossible when $d<d_{\mathrm{CM}}X_1(n)$. We state the resulting criterion for $d=5$.

\begin{proposition}\label{prop:global}
Let $n$ be an integer with $3\nmid n$, and let $a\in(\torz n)^\times$ be such that
\begin{enumerate}
\item $\chi_f(a)=1$ for every newform $f$ of weight $2$ and level $n'$ dividing $n$ with
      $L(f,1)=0$, where $\chi_f$ denotes the nebentypus of $f$, evaluated at the image of $a$
      in $(\torz{n'})^\times$;
\item $b(n)>5$, where $b(n)$ is as in \eqref{eq:bn} with $q=3$ and with $k_3$ determined by
      $a$ as above;
\item $\ord\dm a\nmid 5$ in $\Delta(n)=(\torz n)^\times/\{\pm1\}$;
\item $d_{\mathrm{CM}}X_1(n)>5$.
\end{enumerate}
Then $X_1(n)$ has no non-cuspidal point of degree $5$. In particular, $\torz{n}$ is not a
quintic torsion group.
\end{proposition}

\begin{proof}
We first show that condition (1) implies that $(\dm a-1)J_1(n)(\Q)$ is finite. The Jacobian $J_1(n)$ is isogenous to
$\bigoplus_{n'\mid n}\bigoplus_{[f]} A_{[f]}^{\sigma_0(n/n')}$, the inner sum being over the
Galois orbits of newforms of level $n'$, and this decomposition is equivariant for the diamond
operators, $\dm a$ acting on $A_{[f]}$ through $\chi_f(a)\in\End(A_{[f]})\otimes\Q$. For an $f$
with $L(f,1)\neq0$ the group $A_{[f]}(\Q)$ is finite by Kato's theorem
\cite[Corollary 14.3]{kato}, and for the remaining $f$ we have $\chi_f(a)=1$ by (1), so $\dm a-1$
annihilates $A_{[f]}$. Hence $(\dm a-1)J_1(n)(\Q)$ is contained in the image of a finite group
under an isogeny, and is finite. The rest is
\cite[Proposition 4.3]{DN} and the corollary following it, applied with $d=5$, $q=3$,
$H=\{\pm1\}$, whose hypotheses are (2), (3) and (4).
\end{proof}

\begin{proposition}\label{prop:from_DN}
The groups $\torz n$ for
\[
n\in\{91,\,95,\,119,\,121,\,125,\,133,\,143,\,169,\,187,\,209,\,221,\,247,\,289,\,323,\,361\}
\]
are not quintic torsion groups.
\end{proposition}

\begin{proof}
We apply \Cref{prop:global} with $q=3$ and the values listed in \Cref{table:global}. Note
that $3\nmid n$ for all fifteen levels. The values of $d_{\mathrm{CM}}X_1(n)$ are taken from
the repository accompanying \cite{CGPS}, and all of them are $>5$. We \gitlink{global_table.py}{compute} the values of
$b(n)$ from \eqref{eq:bn}, and all of them are $>5$. We
compute the orders $\ord\dm a$ in $\Delta(n)$, and none of them divides $5$. We \gitlink{char_kernels_all.m}{check}
condition (1) using the LMFDB \cite{lmfdb}. For
$n\in\{119,121,125,169,221,289,361\}$ every newform of weight $2$ and level dividing $n$ with
positive analytic rank has trivial character, so (1) is automatic. For the remaining $n$ the
newforms of positive analytic rank and non-trivial character have character orbits
\lmfdcharorbit{91}{g}, \lmfdcharorbit{91}{u}, \lmfdcharorbit{95}{l},
\lmfdcharorbit{133}{f}, \lmfdcharorbit{133}{s}, \lmfdcharorbit{143}{h},
\lmfdcharorbit{187}{g}, \lmfdcharorbit{209}{f}, \lmfdcharorbit{247}{be} and
\lmfdcharorbit{323}{k}, and the element $a$ of \Cref{table:global} lies in the kernel of
all of them. All analytic ranks used are recorded in the LMFDB as proved.
\end{proof}

\begin{table}[H]
\centering
\begin{tabular}{rcccccc}
\toprule
$n$ & $a \bmod n$ & $\ord\dm a$ & $[\PSL_2(\Z):\overline{\Gamma_1(n)}]$ & $k_3$ & $b(n)$ & $d_{\mathrm{CM}}X_1(n)$\\
\midrule
$91$  & $3$   & $6$   & $4032$  & $7$ & $5.713$  & $24$\\
$95$  & $11$  & $3$   & $4320$  & $8$ & $5.356$  & $72$\\
$119$ & $3$   & $48$  & $6912$  & $7$ & $9.794$  & $96$\\
$121$ & $2$   & $55$  & $7260$  & $8$ & $9.001$  & $110$\\
$125$ & $3$   & $50$  & $7500$  & $7$ & $10.627$ & $50$\\
$133$ & $64$  & $3$   & $8640$  & $8$ & $10.712$ & $36$\\
$143$ & $67$  & $12$  & $10080$ & $8$ & $12.497$ & $120$\\
$169$ & $3$   & $39$  & $14196$ & $7$ & $20.114$ & $52$\\
$187$ & $122$ & $16$  & $17280$ & $8$ & $21.423$ & $160$\\
$209$ & $34$  & $18$  & $21600$ & $8$ & $26.779$ & $180$\\
$221$ & $3$   & $48$  & $24192$ & $7$ & $34.277$ & $96$\\
$247$ & $144$ & $3$   & $30240$ & $8$ & $37.491$ & $72$\\
$289$ & $3$   & $136$ & $41616$ & $7$ & $58.965$ & $136$\\
$323$ & $101$ & $18$  & $51840$ & $8$ & $64.270$ & $288$\\
$361$ & $3$   & $171$ & $64980$ & $7$ & $92.069$ & $114$\\
\bottomrule
\end{tabular}
\medskip
\caption{Data for \Cref{prop:global}: $b(n)$ is rounded to three decimal places, and
$\ord\dm a$ is computed in $\Delta(n)$. It agrees with the order recorded in \cite[Tables 2 and 3]{DN}, except for $n=289$, where \cite[Table 2]{DN} records the order $272$ of $3$ in $(\torz{289})^\times$. The elements $a$ are those of \cite[Tables 2 and 3]{DN}. All of them except $a=2$ at $n=121$ and $a=67$ at $n=143$ are powers of $3$ reduced modulo $n$.}
\label{table:global}
\end{table}

\begin{remark}\label{rem:erratum}
The cases $n\in\{91,121,143,169,187,221,289\}$ of \Cref{prop:from_DN} are the ones covered by
\cite[Remark 6.2]{DN}, and the cases $n\in\{95,119,125,133,209,247,323,361\}$ are
\cite[Theorem 7.1]{DN}. Both are stated there for all degrees $d\leq\lceil b(n)\rceil$. The
correct range is one smaller, $d\leq\lfloor b(n)\rfloor$.

The criterion behind both is \cite[Proposition 4.3]{DN}, whose hypothesis
\cite[(4.1)]{DN} reads $d<b(n)$. As no $b(n)$ occurring in \cite[Tables 2 and 3]{DN} is an
integer, this is the range $d\leq\lfloor b(n)\rfloor=\lceil b(n)\rceil-1$.

For $d=5$ this matters only when $\lceil b(n)\rceil=5$, which happens for $n=77$ and
$n=85$, where $b(77)=\tfrac{14625}{3584}=4.081\ldots$ and $b(85)=\tfrac{8775}{2048}
=4.285\ldots$ Thus the global method does not reach $\torz{77}$ and $\torz{85}$ over quintic fields, and we eliminate them in \Cref{prop:77_85} by a different argument. Every $b(n)$ in
\Cref{table:global} exceeds $5$, so \Cref{prop:from_DN} is unaffected.
\end{remark}

\subsection{Rank zero, cuspidal reduction at \texorpdfstring{$2$}{2}}
The remaining four groups from \cite[\S 7]{DN} are handled by \cite[Proposition 7.3]{DN}.

\begin{proposition}[{\cite[Proposition 7.3]{DN}}]\label{prop:49_51_55_75}
The groups $\torz n$ for $n\in\{49,51,55,75\}$ are not quintic torsion groups.
\end{proposition}

\begin{proof}
We have $\rk J_1(n)(\Q)=0$ by \cite[Theorem 3.1(2)]{DEHMZ}, no elliptic curve over
$\F_{2^k}$ with $k\leq 5$ has a point of order $n$ by the Hasse bound, and Abramovich's bound
gives $\gon_\Q X_1(n)\geq 12,12,15,24$ respectively, all $>10=2\cdot 5$. Now apply
\cite[Lemma 5.3 and Proposition 5.2\,(b)(i)]{DN} with $p=2$ and $d=5$.
\end{proof}

\Cref{prop:from_DN,prop:49_51_55_75} together eliminate the nineteen groups $\torz n$ with
\[
n\in\{49,51,55,75,91,95,119,121,125,133,143,169,187,209,221,247,289,323,361\}.
\]
Once we have shown in \Cref{sec:sieve,sec:direct} that $X_1(2,2n)$ has no non-cuspidal quintic
point for $n=10,11,12$, \Cref{lem:4n} adds $\torz{40},\torz{44},\torz{48}$, and
\Cref{prop:36} will add $\torz{36}$, which accounts for the $23$ groups of \cite[\S 7]{DN}. The
remaining $22$ groups are the subject of \Cref{sec:rank0,sec:sieve,sec:posrank,sec:direct}.

\section{The rank zero criterion}\label{sec:rank0}

We recall \cite[Proposition 5.2]{DN} in the form in which we use it.

\begin{proposition}[{\cite[Proposition 5.2]{DN}}]\label{prop:main}
Let $m\mid n$ with $mn>4$, let $p\nmid mn$ be a prime and let $d\geq 1$. Assume
$\rk J_1(m,n)(\Q(\zeta_m))=0$ and that one of the following holds:
\begin{enumerate}
\item[\textup{(a)}] $p>2$ and $\gon_{\Q(\zeta_m)}X_1(m,n)>d$;
\item[\textup{(b.i)}] $p=2$ and $\gon_{\Q(\zeta_m)}X_1(m,n)>2d$;
\item[\textup{(b.ii)}] $p=2$, $\gon_{\Q(\zeta_m)}X_1(m,n)>d$, and there are a field $L$ with
$\Q(\zeta_m)\subseteq L\subseteq\Q(\zeta_n)$ and a prime $\wp_L$ of $L$ above $2$ such that
every cuspidal $\F_{\wp_L}$-rational divisor on $X_1(m,n)$ lifts to an $L$-rational cuspidal
divisor and reduction modulo $\wp_L$ is injective on $J_1(m,n)(L)[2]$.
\end{enumerate}
Then every elliptic curve $E$ over a degree $d$ extension $K$ of $\Q(\zeta_m)$ with
$\tg{m}{n}\subseteq E(K)$ has good reduction at at least one prime of $K$ above $p$.
\end{proposition}

Combining \Cref{prop:main} with \Cref{lem:badred} eliminates four groups outright.

\begin{proposition}\label{prop:pure}
The groups $\torz{27}$, $\torz{35}$, $\torz{39}$ and $\torz{50}$ are not quintic torsion
groups.
\end{proposition}

\begin{proof}
Let $n\in\{27,35,39,50\}$, and suppose for contradiction that $E$ is an elliptic curve over a
quintic field $K$ with a point of order $n$. In all four cases $\rk J_1(n)(\Q)=0$ by
\cite[Theorem 3.1(2)]{DEHMZ}.

Take $p=2$ for $n=27,35,39$ and $p=3$ for $n=50$. By enumerating all elliptic curves over
$\F_{p^k}$ for $k\leq 5$ we \gitlink{torsion_over_Fq.m}{find} that none of them has a point of order $n$,
so by \Cref{lem:badred} the curve $E$ has good reduction at
no prime of $K$ above $p$. Hence it is enough to verify the hypotheses of \Cref{prop:main}
with $d=5$, as its conclusion is that $E$ has good reduction at some prime of $K$ above $p$.

For $n=35$ and $n=39$ we have $\gon_\Q X_1(35)=12$ and $\gon_\Q X_1(39)=14$ by
\cite[Theorem 4 and Table 1]{DvH}, both $>10=2\cdot 5$, so \Cref{prop:main}(b.i) applies.

For $n=50$, Abramovich's bound with $[\PSL_2(\Z):\overline{\Gamma_1(50)}]=900$ gives
$\gon_\Q X_1(50)\geq \gon_\C X_1(50)>8.9$, so $\gon_\Q X_1(50)>5$ and \Cref{prop:main}(a)
applies.

For $n=27$ we use \Cref{prop:main}(b.ii) with $L=\Q$. All cusps of $X_1(27)$ are defined over
subfields of $\Q(\zeta_{27})$, in which $2$ is totally inert ($2$ is a primitive root modulo
$27$), so every $\F_2$-rational cuspidal divisor is the reduction of a $\Q$-rational one.
Finally $J_1(27)(\Q)=\ClCusp_\Q X_1(27)\simeq \torz3\times\torz3\times\torz{52497}$ by
\cite[Corollary 4.14 and Table 2]{DEHMZ} has no $2$-torsion, so reduction is trivially injective
on $J_1(27)(\Q)[2]$. As $\gon_\Q X_1(27)=6>5$ by \cite[Theorem 4 and Table 1]{DvH},
\Cref{prop:main}(b.ii) applies.
\end{proof}

\begin{remark}
Note that some of the enumerations in \Cref{prop:pure} can be avoided by using \cite[Theorem 4.1]{waterhouse} instead.
\end{remark}

\section{The Hecke sieve in degree five}\label{sec:sieve}

Let $X=X_1(m,n)$ with $m\in\{1,2\}$ and $\rk J(\Q)=0$, let $p\nmid 2n$ be a prime, let $K$ be a
quintic field and let $x\in Y_1(m,n)(K)$ be a non-cuspidal point, corresponding to an elliptic
curve $E/K$ together with a level structure $\tg{m}{n}\hookrightarrow E(K)$, so that
$x^{(5)}:=\sum_{\sigma}x^\sigma\in X^{(5)}(\Q)$. Reducing modulo $p$ we may write
\[
x^{(5)}_{\F_p}=D+C,\qquad D\in Y_1(m,n)^{(d')}(\F_p),\quad C \text{ cuspidal of degree }5-d',
\]
where $d'\geq 1$ if and only if $E$ has good reduction at some prime of $K$ above $p$. We call
the multiset of degrees of the closed points occurring in $D$, a partition of $d'$, the
\emph{type} of the decomposition, and write it in brackets. Thus $[5]$ means that $D$ is a
single closed point of degree $5$, and $[3,2]$ that $D$ is the sum of one of degree $3$ and one
of degree $2$. Any type with $d'=5$ leaves no cuspidal complement. 

The Hecke sieve \cite[Proposition 5.5]{DN} says that if $[A_q(D)]\neq 0$ and $[A_q(C)]=0$ in
$J(\F_p)$ for some prime $q\nmid 2pn$, then this decomposition cannot occur. Since $A_q$
commutes with the diamond operators and these are automorphisms of $X$, it suffices to test one
$D$ in each orbit of the diamond action.

The hypothesis $[A_q(C)]=0$ needs more care in degree $5$ than in degree $4$, because the
cuspidal complement $C$ can now have degree up to $4$. It is supplied by
\cite[Lemma 5.10]{DN}, which we restate below in the form in which our code checks it and in
the form needed for the twisted operators of \Cref{sec:posrank}. We first fix notation for the
cusps.

The aim is to lift a cuspidal divisor on $X_{\F_p}$ to a cuspidal divisor over a number field
in which both $p$ and $q$ split completely, so that \Cref{prop:tools}(2) applies.

By \cite[Lemma 2.8]{DEHMZ} the cuspidal subscheme $\mathcal C$ of $X_1(n)$ over $\Z[1/2n]$ is
$\bigsqcup_{d\mid n}\bigl(\mu'_{n/d}\times(\torz d)'\bigr)/[-1]$, where the primes select the
elements of maximal order in \emph{each} factor separately, so that the piece indexed by $d$ has
$\varphi(d)\varphi(n/d)/2$ geometric points. In particular, $\mathcal C$ is finite \'etale over
$\Z[1/2n]$. It follows that reduction modulo $p$ is a bijection
$\mathcal C(\overline{\Q}_p)\to\mathcal C(\overline{\F}_p)$, equivariant
for the surjection
$\Gal(\overline{\Q}_p/\Q_p)\twoheadrightarrow\Gal(\overline{\F}_p/\F_p)$. All cusps are defined
over $\Q(\zeta_n)$ and $p$ is unramified there, so the Galois action on $\mathcal C(\Qbar)$
factors through $(\torz n)^\times$, and the action of Frobenius on the reductions is that of
$D_p$, the subgroup generated by $p\bmod n$. It follows that the closed points of
$\mathcal C_{\F_p}$ correspond to the $D_p$-orbits on $\mathcal C(\Qbar)$.

Let $C_i$ be such a closed point, let $O_i$ be the corresponding $D_p$-orbit, and set
\[
H_i:=\{\sigma\in(\torz n)^\times:\sigma(O_i)=O_i\}\supseteq D_p,
\qquad L_i:=\Q(\zeta_n)^{H_i}.
\]
The divisor $c_i:=\sum_{P\in O_i}P$ is rational over $L_i$ and reduces to $C_i$. As
$\Q(\zeta_n)/\Q$ is abelian, so is $L_i/\Q$, with $\Gal(L_i/\Q)=(\torz n)^\times/H_i$ and
$\Frob_q=q\bmod H_i$. Hence $p$ splits completely in $L_i$, and $q$ splits completely in $L_i$
if and only if $q\bmod n\in H_i$.

\begin{lemma}\label{lem:CC}
Let $X=X_1(n)$ with $n>4$, and let $p\nmid 2n$ and $q\nmid 2pn$ be primes.
\begin{enumerate}
\item Let $L\subseteq\Q(\zeta_n)$ be a subfield in which both $p$ and $q$ split completely, and
let $C$ be a cuspidal divisor on $X_{\F_p}$ that is the reduction of a cuspidal divisor on
$X_L$. Then $[A_q\psi(C)]=0$ in $J(\F_p)$ for every $\psi\in\Z[\Delta(n)]$.
\item Let $C=\sum_i a_iC_i$ be a cuspidal divisor on $X_{\F_p}$, the $C_i$ being pairwise
distinct closed points. If $q\bmod n$ lies in $H_i$ for every $i$, then $[A_q\psi(C)]=0$ for
every $\psi\in\Z[\Delta(n)]$. This holds in particular whenever $q\equiv p^j\pmod n$ for some
$j\geq0$.
\end{enumerate}
The case $\psi=1$ of \textup{(2)} is \cite[Lemma 5.10]{DN}.
\end{lemma}

\begin{proof}
(1) Let $c$ be a cuspidal divisor on $X_L$ reducing to $C$ and put $e:=\deg\psi(c)$. The
diamond operators are defined over $\Q$ and permute the cusps, so $\psi(c)$ is again a cuspidal
divisor rational over $L$, and $\psi(c)-ec_0$ is a cuspidal divisor of degree $0$ on $X_L$. Its
class lies in $J(L)_\tors$ by Manin--Drinfeld. As $q$ splits completely in $L$,
\Cref{prop:tools}(2) gives $[A_q(\psi(c)-ec_0)]=0$, and $A_q(c_0)=0$ by \Cref{prop:tools}(3),
so $[A_q\psi(c)]=0$ in $J(L)$. Reducing modulo a prime of $L$ above $p$, whose residue field is
$\F_p$ because $p$ splits completely in $L$, and using that $A_q$ and $\psi$ commute with
reduction, gives $[A_q\psi(C)]=0$.

(2) Both $\psi$ and $A_q$ are linear on divisors, so $[A_q\psi(C)]=\sum_i a_i\,[A_q\psi(C_i)]$ and it suffices to prove $[A_q\psi(C_i)]=0$ for all $i$. Fix $i$ and apply (1) with $L=L_i$ and $C=C_i$: the divisor $c_i$ is cuspidal, defined over $L_i$ and reduces to $C_i$, the prime $p$ splits completely in $L_i$ because $D_p\subseteq H_i$, and $q$ splits completely in $L_i$ because $q\bmod n\in H_i$ by hypothesis.
For the last assertion, $q\equiv p^j\pmod n$ means $q\bmod n\in D_p\subseteq H_i$ for every $i$.

Since $A_q$ commutes with $\Z[\Delta(n)]$ and $A_q(C)$ has degree $0$, we have
$[A_q\psi(C)]=\psi\bigl([A_q(C)]\bigr)$, so the case $\psi=1$ implies the general statement.
\end{proof}

\begin{remark}
Only the closed points that actually occur in a cuspidal complement matter, so in practice one
only has to impose $q\bmod n\in H_i$ for the closed cuspidal points $C_i$ of degree at most
$5-d'$. Moreover $H_i=(\torz n)^\times$ (so that $C_i$ imposes no condition on $q$) as soon as
the $D_p$-orbit $O_i$ is a full $(\torz n)^\times$-orbit, i.e.\ as soon as the corresponding
Galois orbit of cusps over $\Q$ stays irreducible modulo $p$. Our code compares the degrees of
the Galois orbits of cusps over $\Q$ with the degrees of their reductions to find the orbits
that split modulo $p$. If a cuspidal complement can contain a point coming from such an orbit,
it only uses $q$ with $q\bmod n\in D_p$, which suffices as $D_p\subseteq H_i$.
\end{remark}

We can now eliminate eight of the remaining groups.

\begin{proposition}\label{prop:sieve}
The groups $\torz n$ for $n\in\{26,32,33,34,38,42,45\}$ and $\tg{2}{22}$ are not quintic
torsion groups. Hence, by \Cref{lem:4n}, $\torz{44}$ is not a quintic torsion group
either.
\end{proposition}

\begin{proof}
Let $X=X_1(m,n)$ be one of the eight curves and let $p$ be the prime in \Cref{table:sieve}. In every case $\rk J(\Q)=0$. For $m=1$ this is \cite[Theorem 3.1(2)]{DEHMZ}, since
$26,32,33,34,38,42,45$ lie in $S_0$ and not in the excluded set, and for $X_1(2,22)$ it is
\cite[Theorem 3.1(3)]{DEHMZ} (or \cite[Theorem 4.1]{DS}).

We first exclude the totally cuspidal reduction, i.e.\ $d'=0$, using \Cref{prop:main}(a) with
$d=5$. This requires $\gon_\Q X>5$, which holds in every case. By \cite[Theorem 4 and Table 1]{DvH}
we have $\gon_\Q X_1(n)=6,8,10,10,12$ for $n=26,32,33,34,38$. For $n=42$, Derickx and van
Hoeij proved $\gon_\Q X_1(42)>8$ \cite[proof of Proposition 6]{DvH}, and for $n=45$,
Abramovich's bound with index $864$ gives $\gon_\Q X_1(45)\geq 9$. For $X_1(2,22)$, Derickx
and Sutherland proved $\gon_{\F_3}X_1(2,22)>5$ \cite[proof of Proposition 5.3]{DS}, so
$\gon_\Q X_1(2,22)\geq\gon_{\F_3}X_1(2,22)>5$.

It remains to exclude every decomposition $x^{(5)}_{\F_p}=D+C$ with $\deg D=d'\geq 1$. The
types that can occur are limited by \Cref{lem:badred}: a non-cuspidal closed point of degree $e$
can occur only if some elliptic curve over $\F_{p^e}$ has the required level structure, and we
determine these degrees by an explicit enumeration. The number of non-cuspidal closed points of each degree,
and the number of their diamond orbits, are listed in \Cref{table:sieve}. For one
representative $D$ of each diamond orbit we \gitlink{hecke_sieve_deg5_v6.m}{compute} that $[A_q(D)]\neq0$ for one of the
admissible primes $q$. Since $A_q$ commutes with the
diamond operators, this covers every $D$. That $[A_q(C)]=0$ for every cuspidal complement $C$
that can occur follows from \Cref{lem:CC}\,(2) with $\psi=1$ for the seven curves $X_1(n)$,
and holds vacuously for $X_1(2,22)$, where the only type occurring is $[5]$ and there is no
cuspidal complement at all. There are no survivors in any type (see \Cref{table:sieve}).
\end{proof}

\begin{table}[H]
\centering
\small
\begin{tabular}{lrr@{\;}l@{\;}l@{\;}l}
\toprule
$X$ & $g$ & $p$ & non-cuspidal points of degree $1,\dots,5$ & orbits & $q$\\
\midrule
$X_1(26)$   & $10$ & $7$ & $0,42,124,603,3312$ & $0,8,22,103,552$ & $3,5,11$\\
$X_1(32)$   & $17$ & $3$ & $0,0,8,8,16$        & $0,0,1,1,2$      & $5,7,11,13$\\
$X_1(33)$   & $21$ & $7$ & $0,0,40,660,3644$   & $0,0,4,68,366$   & $5,13,19,37,43,67$\\
$X_1(34)$   & $21$ & $3$ & $0,0,0,26,32$       & $0,0,0,5,4$      & $5,7,11,13$\\
$X_1(38)$   & $28$ & $3$ & $0,0,3,0,72$        & $0,0,1,0,8$      & $5,7,11,13$\\
$X_1(42)$   & $25$ & $5$ & $0,0,4,183,648$     & $0,0,2,33,108$   & $11,13,17,19,23,29$\\
$X_1(45)$   & $41$ & $7$ & $0,24,120,348,3276$ & $0,2,10,30,273$  & $11,13,17,19$\\
$X_1(2,22)$ & $16$ & $3$ & $0,0,0,0,60$        & $0,0,0,0,12$     & $5,7,13$\\
\bottomrule
\end{tabular}
\medskip
\caption{The Hecke sieve computations of \Cref{prop:sieve}. The fourth column lists the
number of non-cuspidal closed points of $X_{\F_p}$ of degree $1,2,3,4,5$, and the fifth the
number of their orbits under the diamond action. Reducible types such as $[2,2]$ and $[3,2]$ are assembled from these and are processed as well. The sieve leaves no survivors in any type for any of the eight curves.}
\label{table:sieve}
\end{table}

\begin{remark}\label{rem:CC_used}
We spell out how \Cref{lem:CC} is applied for each curve.
\begin{itemize}
\item For $(n,p)=(26,7),(34,3),(38,3)$ the prime $p$ is a primitive root modulo $n$, so
$D_p=(\torz n)^\times$ and every $q\nmid 2pn$ is admissible.
\item For $(n,p)=(32,3)$ the subgroup of $(\torz{32})^\times$ generated by $3$ has order $8$
and contains $11$, so we ran the types $[3]$ and $[4]$, which have a cuspidal
complement, with $q=11$, and the type $[5]$ with $q\in\{5,7,11,13\}$.
\item For $X_1(2,22)$ at $p=3$ only the type $[5]$ occurs, so there is no cuspidal complement
and no condition on $q$.
\item For $X_1(33)$ at $p=7$ the Galois orbits of cusps over $\Q$ have degrees $1^{10},2^5,10^2$,
and modulo $7$ the five orbits of degree $2$ split into rational cusps while the two of degree
$10$ stay irreducible. The condition becomes $q\bmod 33\in D_7=\{1,4,7,10,13,16,19,25,28,31\}$, so we ran the
types $[3]$ and $[4]$ with $q\in\{13,19,37,43,67\}$ and the type $[5]$ with
$q\in\{5,13,19,37,43,67\}$.
\item For $X_1(42)$ at $p=5$ the orbits have degrees $1^{12},2^6,6^4$, and modulo $5$ two of
the four orbits of degree $6$ split into two places of degree $3$ each while everything else
stays irreducible. So complements of degree at most $2$ impose no condition on $q$, and
complements of degree $3$ or $4$ do not occur, because $Y_1(42)(\F_{5^e})=\emptyset$ for
$e\leq2$.
\item For $X_1(45)$ at $p=7$ the orbits have degrees $1^{12},2^4,4^3,6^2,8,12$, and modulo $7$ the
four orbits of degree $2$ split into rational cusps, the two of degree $6$ into places of degree
$3$ and the one of degree $8$ into places of degree $4$. A cuspidal complement containing one of
these points requires $q\bmod 45\in D_7=\{1,4,7,13,16,19,22,28,31,34,37,43\}$, that is
$q\equiv1\pmod3$. We imposed this condition on every cuspidal complement, so we ran the types
$[2],[3],[4],[2,2]$ with $q\in\{13,19\}$ and the types $[5],[3,2]$, which have no cuspidal
complement, with $q\in\{11,13,17,19\}$.
\end{itemize}
For $X_1(33)$, $X_1(42)$ and $X_1(45)$ we computed the condition from the model, checking also that the degrees of the reductions of the cusps agree with the prediction of \cite[Lemma 2.8]{DEHMZ}. For the other four curves $X_1(n)$ we imposed the coarser condition $q\bmod n\in D_p$ instead, which by \Cref{lem:CC} is always sufficient.
\end{remark}

\section{The cases of positive analytic rank}\label{sec:posrank}

We now treat the groups $\torz n$ for $n\in\{57,63,65,77,85\}$. For all five, $J_1(n)$ has a newform factor of positive analytic rank, so Kato's theorem does not give
$\rk J_1(n)(\Q)=0$ and neither \Cref{prop:main} nor the Hecke sieve of \Cref{sec:sieve} applies.
What we use instead is the finiteness of $(\dm a-1)J_1(n)(\Q)$ for suitable $a$. A diamond
argument applied to a totally cuspidal reduction eliminates $n=77$ and $n=85$, while for $n=57,63,65$ we combine that argument at an odd prime with a twisted Hecke sieve. Both
arguments are unconditional, and neither uses any information about $\rk J_1(n)(\Q)$.

Throughout this section we use the following consequence of Kato's theorem, already used in the
proof of \Cref{prop:global}: if $a\in(\torz n)^\times$ satisfies $\chi_f(a)=1$ for every
newform $f$ of weight $2$ and level dividing $n$ with $L(f,1)=0$, then $(\dm a-1)J_1(n)(\Q)$ is
finite. We call such an $a$ \emph{admissible for $n$}. \Cref{table:posrank} lists, for each
$n\in\{57,63,65,77,85\}$, the newforms of positive analytic rank, taken from the LMFDB
\cite{lmfdb} (all analytic ranks used are recorded there as proved).

\begin{table}[H]
\centering
\begin{tabular}{rlll}
\toprule
$n$ & newforms of positive analytic rank & admissible $a$ used & $\ord\dm a$\\
\midrule
$57$ & \lmfdbnewform{57}{2}{a}{a} (trivial character) & $2$, $4$ & $9$, $9$\\
$63$ & \lmfdbnewform{63}{2}{s}{a} ($\chi$ of order $6$) & $5$ & $3$\\
$65$ & \lmfdbnewform{65}{2}{a}{a} (trivial character) & $2$, $3$ & $6$, $12$\\
$77$ & \lmfdbnewform{77}{2}{a}{a} (trivial character) & $3$ & $30$\\
$85$ & \lmfdbnewform{85}{2}{a}{b}, \lmfdbnewform{85}{2}{j}{a} ($\chi$ of order $4$) & $2$ & $8$\\
\bottomrule
\end{tabular}
\medskip
\caption{The newforms of positive analytic rank. For $n=57,65,77$ every $a$ is admissible, because the only
newforms of positive analytic rank have trivial character. For $n=63$ the kernel of $\chi$ is
$\{1,5,25,38,58,62\}$, whose image in $\Delta(63)$ is the subgroup generated by $\dm 5$, of order $3$. For $n=85$ the kernel of $\chi$
contains $2$, which is the choice made in \cite[Table 2]{DN}.}
\label{table:posrank}
\end{table}

\begin{remark}\label{rem:whynotglobal}
The global method of \Cref{prop:global} does not reach these five cases. For every odd prime
$q\nmid n$ one gets $b(n)<5$, as soon as $J_1(n)(\Q)$ is infinite. Only $q=2$, and only for
$n=77$ and $85$, would give $b(n)>5$, but there \cite[Proposition 4.3]{DN} has the extra
assumption on the $2$-primary part of $(\dm a-1)J_1(n)(\Q)$ recorded in \Cref{sec:literature},
which we have not verified.
\end{remark}

\subsection{A diamond argument for the totally cuspidal reduction}

\begin{proposition}\label{prop:M3}
Let $n>4$ be an integer, let $p\nmid n$ be a prime, and let $a\in(\torz n)^\times$ be admissible
for $n$. Set $\varepsilon:=1$ if $p>2$, and $\varepsilon:=2$ if $p=2$ and $n$ is odd. Assume
\begin{enumerate}
\item $\ord\dm a\nmid 5$ in $\Delta(n)$;
\item $\gon_\C X_1(n)>10\varepsilon$;
\item $d_{\mathrm{CM}}X_1(n)>5$.
\end{enumerate}
Let $K$ be a quintic field and $E/K$ an elliptic curve with a point of order $n$, giving
$x\in Y_1(n)(K)$, and assume $[\Q(x):\Q]=5$. Then $E$ has good reduction at at least one prime
of $K$ above $p$, i.e. $x^{(5)}_{\F_p}$ has a non-cuspidal component.
\end{proposition}

\begin{proof}
Suppose the opposite. By \Cref{lem:additive} the reduction of $E$ at every prime of $K$ above $p$ is
multiplicative, so, fixing a prime $\wp$ of $\Q(\zeta_n)$ above $p$, the reduction of $x^{(5)}$
modulo $\wp$ is an effective divisor of degree $5$ supported on the cusps. As in the proof of \cite[Proposition 5.2]{DN}, it lifts to an effective cuspidal divisor $C$ of
degree $5$ on $X_1(n)_{\Q(\zeta_n)}$, since the cuspidal subscheme is finite flat over
$\Z[1/n]$ and every cusp of $X_1(n)$ is $\Q(\zeta_n)$-rational.

Let $y:=(\dm a-1)[x^{(5)}-C]\in J_1(n)(\Q(\zeta_n))$. Writing
$[x^{(5)}-C]=[x^{(5)}-5c_0]+[5c_0-C]$ and using that $[x^{(5)}-5c_0]\in J_1(n)(\Q)$, that
$(\dm a-1)J_1(n)(\Q)$ is finite, and that $[5c_0-C]$ is torsion by Manin--Drinfeld, we see that
$y$ is torsion. By construction $y$ reduces to $0$ modulo $\wp$. If $p>2$, reduction is
injective on $J_1(n)(\Q(\zeta_n))_\tors$ by \cite[Appendix]{katz}, so $y=0$. If $p=2$ and $n$ is odd, the kernel of reduction has exponent at most $2$ by \cite[Proposition 2.4]{parent}, so $2y=0$. In either case $\varepsilon\, y=0$, i.e.
\[
\varepsilon(\dm a-1)(x^{(5)}-C)=\ddiv(f)
\]
for some $f\in\Q(\zeta_n)(X_1(n))$.

If $f$ is non-constant then, separating the positive and negative parts,
$\deg f\leq \varepsilon\deg\bigl(\dm a x^{(5)}+C\bigr)=10\varepsilon$, contradicting (2).

So $f$ is constant, i.e.\ $\varepsilon(\dm a-1)(x^{(5)}-C)=0$ as a divisor. The diamond
operators preserve the set of cusps and its complement, so
$\dm a x^{(5)}=x^{(5)}$ and $\dm aC=C$. Now $\dm a$ permutes the five distinct
geometric points in the support of $x^{(5)}$. If every one of them had trivial stabilizer in
the cyclic group $\langle\dm a\rangle$, all orbits would have length $\ord\dm a$ and hence
$\ord\dm a\mid 5$, contradicting (1). So some conjugate $x^\sigma$ satisfies
$\dm b x^\sigma=x^\sigma$ with $b\not\equiv\pm1\pmod n$. Writing $x^\sigma=(E',P')$, this means
$(E',P')\simeq(E',bP')$, so $E'$ has an automorphism $\zeta\neq\pm1$, hence $j(E')\in\{0,1728\}$
and $x^\sigma$ is a CM point of degree $5$, contradicting (3).
\end{proof}

\begin{proposition}\label{prop:77_85}
The groups $\torz{77}$ and $\torz{85}$ are not quintic torsion groups.
\end{proposition}

\begin{proof}
Let $n\in\{77,85\}$ and suppose for contradiction that $E$ is an elliptic curve over a quintic
field $K$ with a point of order $n$. We apply \Cref{prop:M3} with $p=2$. Since $n$ is odd in both cases, $\varepsilon=2$. By the Hasse bound no elliptic curve over $\F_{2^k}$ with
$k\leq 5$ has a point of order $n$ for $n\geq 45$ as $(2^{5/2}+1)^2<45$. So by \Cref{lem:badred} the curve $E$ has bad reduction at all
primes of $K$ above $2$, contradicting the conclusion of \Cref{prop:M3} once its assumptions are verified.

For $n=77$ take $a=3$, which is admissible because the only newform of positive analytic rank
of level dividing $77$ has trivial character. We have $\ord\dm 3=30\nmid 5$. Abramovich's bound
with $[\PSL_2(\Z):\overline{\Gamma_1(77)}]=2880$ gives $\gon_\C X_1(77)>28.56>20$, and
$d_{\mathrm{CM}}X_1(77)=60>5$.

For $n=85$ take $a=2$ (as in \cite[Table 2]{DN}). It lies in the kernel of the
character of \lmfdbnewform{85}{2}{j}{a}, the only newform of positive analytic rank and
non-trivial character of level dividing $85$, so $a=2$ is admissible, and
$\ord\dm 2=8\nmid5$. Abramovich's bound with index $3456$ gives $\gon_\C X_1(85)>34.27>20$, and
$d_{\mathrm{CM}}X_1(85)=32>5$.
\end{proof}

\begin{remark}
For $n=57,63,65$ Abramovich's bound gives $\gon_\C X_1(n)>14.28$, $17.14$, $19.99$
respectively. These do not prove $\gon_\C X_1(n)>20$, so they do not allow us to apply
\Cref{prop:M3} with $p=2$. For $n=65$ the bound gives only $\gon_\C X_1(65)\geq 20$, whereas the criterion
requires $\gon_\C X_1(65)>20$. Under Selberg's eigenvalue conjecture $\lambda_1\geq1/4$ it would
read $\gon_\C X_1(65)\geq 2016/96=21$ and $\torz{65}$ would be excluded by \Cref{prop:M3} alone.
Thus we work at an odd prime, where $\varepsilon=1$ and only $\gon_\C X_1(n)>10$ is needed.
\end{remark}

\subsection{The twisted Hecke sieve}

For the remaining three groups we combine \Cref{prop:M3} at an odd prime (which excludes the
totally cuspidal reduction) with a variant of the Hecke sieve in which $A_q$ is replaced by an
operator that kills $J_1(n)(\Q)$.

\begin{proposition}[Twisted Hecke sieve]\label{prop:twisted}
Let $n>4$, let $p\nmid 2n$ and $q\nmid 2pn$ be primes, and let $\psi\in\Z[\Delta(n)]$ be such
that $\psi\bigl(J_1(n)(\Q)\bigr)$ is finite. Put $t:=A_q\psi$. Let $D+C\in X_1(n)^{(5)}(\F_p)$
with $D$ effective non-cuspidal of
degree $d'\geq1$ and $C$ effective cuspidal of degree $5-d'$, and suppose that
\begin{enumerate}
\item $[t(D)]\neq0$ in $J_1(n)(\F_p)$, and
\item $q\bmod n$ lies in $H_i$ for every closed point $C_i$ in the support of $C$, in the
notation of \Cref{lem:CC}.
\end{enumerate}
Then no $z\in Y_1(n)^{(5)}(\Q)$ reduces to $D+C$.
\end{proposition}

\begin{proof}
Suppose some $z\in Y_1(n)^{(5)}(\Q)$ reduces to $D+C$. Since $A_q$ commutes with the diamond
operators and $A_q(c_0)=0$, we have $t(c_0)=0$, so $[t(z)]=A_q\bigl(\psi([z-5c_0])\bigr)$. Now
$\psi([z-5c_0])$ lies in $\psi\bigl(J_1(n)(\Q)\bigr)$ and is hence torsion, so $A_q$ kills it by \Cref{prop:tools}(2) applied with $L=\Q$. It follows that $[t(z)]=0$ in $J_1(n)(\Q)$.

As $\psi\in\Z[\Delta(n)]$, hypothesis (2) and \Cref{lem:CC}\,(2) give $[t(C)]=0$ in
$J_1(n)(\F_p)$. Reducing modulo $p$ now gives $0=[t(z)]_{\F_p}=[t(D)]+[t(C)]=[t(D)]$,
contradicting (1).
\end{proof}

We use two instances of $t$:
\[
t_a:=A_q(\dm a-1),\qquad
t_H:=A_q\pi^*\pi_*=A_q\Bigl(\textstyle\sum_{h\in H/\{\pm1\}}\dm h\Bigr).
\]
Here $a$ is admissible for $n$, so $\psi=\dm a-1$ satisfies the hypothesis by definition. Furthermore, $H\subseteq(\torz n)^\times$ is a subgroup containing $-1$ with $\rk J_H(\Q)=0$ and $\pi:X_1(n)\to X_H$ is the quotient map, of degree $\#H/2$ since $\dm{-1}$ acts trivially, so $\psi=\pi^*\pi_*$ satisfies the hypothesis because $\pi_*J_1(n)(\Q)\subseteq J_H(\Q)$ is finite.

\begin{remark}
A divisor $D$ with $\dm aD=D$ satisfies $t_a(D)=0$ and hence survives the sieve trivially. Such a $D$ has to be treated with a different $a$, or with $t_H$.
\end{remark}

\begin{proposition}\label{prop:57_63_65}
The groups $\torz{57}$, $\torz{63}$ and $\torz{65}$ are not quintic torsion groups.
\end{proposition}

\begin{proof}
Let $n\in\{57,63,65\}$ and let $p=5,5,3$ respectively. Suppose $E$ is an elliptic curve over a
quintic field $K$ with a point of order $n$, and let $x\in Y_1(n)(K)$ be the corresponding
point. The proof has two steps: \Cref{prop:M3} shows that the reduction of $x^{(5)}$ modulo $p$ is not totally cuspidal, and \Cref{prop:twisted} then excludes every remaining possibility.

\emph{The totally cuspidal reduction.} We apply \Cref{prop:M3} with the odd prime $p$, so
$\varepsilon=1$ and we need $\gon_\C X_1(n)>10$, $\ord\dm a\nmid5$ and
$d_{\mathrm{CM}}X_1(n)>5$. Abramovich's bound gives $\gon_\C X_1(n)>14.28,17.14,19.99$
respectively, all $>10$. We take $a=2,5,2$ respectively, which are admissible by
\Cref{table:posrank}, and $\ord\dm a=9,3,6$, none of which divides $5$. Finally
$d_{\mathrm{CM}}X_1(n)=12,36,24$, all $>5$. Hence $x^{(5)}_{\F_p}$ has a non-cuspidal
component $D$ of some degree $d'\geq1$.

\emph{The remaining reductions.} We rule out every possible $D+C$ by \Cref{prop:twisted}.
By enumerating the elliptic curves over $\F_{p^e}$ we find that $Y_1(n)(\F_{p^e})$ is empty
for every $e\leq5$ except those listed in \Cref{table:twisted}, so $D$ is irreducible of one
of the degrees listed there. The numbers of non-cuspidal closed points of $X_1(n)_{\F_p}$ of
each of those degrees, and of their orbits under the diamond action, are given in the table.
For one representative $D$ of each orbit we \gitlink{hecke_sieve_twisted_v6.m}{compute} that $[t(D)]\neq0$ for one of the listed
primes $q$, and there are no survivors. As $t$
commutes with the diamond operators, this covers every $D$.

Condition (2) of \Cref{prop:twisted} is vacuous in all three cases, as we now explain. By \cite[Lemma 2.8]{DEHMZ} the Galois orbits of cusps over $\Q$ have degrees
$1^{18},2^{9},18,18$ for $n=57$ (fields $\Q$, $\Q(\zeta_3)$, $\Q(\zeta_{19})$,
$\Q(\zeta_{57})^+$), $1^{18},2^{6},6^{3},6^{3},12,18$ for $n=63$ ($\Q$, $\Q(\zeta_3)$,
$\Q(\zeta_7)$, $\Q(\zeta_9)$, $\Q(\zeta_{21})$, $\Q(\zeta_{63})^+$) and
$1^{24},4^{6},12^{2},24$ for $n=65$ ($\Q$, $\Q(\zeta_5)$, $\Q(\zeta_{13})$,
$\Q(\zeta_{65})^+$). Only the types with a cuspidal complement of degree at most $2$ occur, and
for those the groups $H_i$ of \Cref{lem:CC} are all of $(\torz n)^\times$: modulo $p$ the
orbits of degree $2$ stay irreducible (for $n=57,63$), so every cuspidal closed point of $X_1(n)_{\F_p}$ of degree $\leq2$ is either the reduction of a $\Q$-rational cusp or the reduction of a full Galois orbit. 

It follows that the assumptions of \Cref{prop:twisted} are satisfied for every possible $D+C$, so no element of $Y_1(n)^{(5)}(\Q)$ reduces to any of them. But $x^{(5)}\in Y_1(n)^{(5)}(\Q)$ does, giving a contradiction. 
\end{proof}

\begin{table}[H]
\centering
\footnotesize
\begin{tabular}{r|r|r|@{\quad }l|@{\ }c|@{\quad }l|@{\quad }l|@{\quad }l}
\hline
$n$ & $g$ & $p$ & $t$ & $e$ with $Y_1(n)(\F_{p^e})\neq\emptyset$ & points & orbits & $q$\\
\hline
$57$ & $85$  & $5$ & $t_a$, $a\in\{2,4\}$ & $3,4,5$ & $72,81,756$  & $4,5,42$ & $7,11,13$\\
$63$ & $97$  & $5$ & $t_5$, $t_{\dm{\pm8}}$ & $3,4,5$ & $12,108,720$ & $2,6,40$ & $11,13,17$\\
$65$ & $121$ & $3$ & $t_a$, $a\in\{2,3\}$ & $4,5$   & $12,96$      & $1,4$    & $7,11$\\
\bottomrule
\end{tabular}
\medskip
\caption{The twisted Hecke sieve computations of \Cref{prop:57_63_65}. The columns ``points''
and ``orbits'' list, for the degrees $e$ in the previous column, the number of non-cuspidal
closed points of $X_1(n)_{\F_p}$ of degree $e$ and the number of their diamond orbits. In each
case the sieve leaves no survivors.}
\label{table:twisted}
\end{table}

\begin{remark}
For $n=63$ the operator $t_H$ with $H=\dm{\pm8}$ is available because $\chi(8)=-1$ for the character $\chi$ of
\lmfdbnewform{63}{2}{s}{a}, so that \lmfdbnewform{63}{2}{s}{a} does not contribute to
$J_H$ with $H=\dm{\pm8}$: the isogeny factors of $J_H$ correspond to the newforms of level
dividing $63$ whose character is trivial on $H$, i.e. of order $1$ or $3$, and all of these
have analytic rank $0$ (this is the argument of \cite[Proposition 6.4]{DN}). We apply $t_5$ first and $t_{\dm{\pm8}}$ to the divisors that $t_5$ does not exclude. No divisor survives both tests.
\end{remark}

\section{Direct analysis via the cuspidal class group}\label{sec:direct}

The Hecke sieve shows that certain divisors on $X_{\F_p}$ are not reductions of rational
divisors, but it says nothing about the divisors it does not eliminate. To determine
\emph{all} quintic points on $X_1(28)$, $X_1(30)$ and $X_1(2,18)$ we use instead a variant of
the ``direct analysis over $\F_p$'' of \cite[\S 5.2]{DEHMZ}, which determines all rational
effective divisors of a given degree.

\begin{proposition}\label{prop:direct}
Let $X=X_1(m,n)$ with $m\in\{1,2\}$, let $p\nmid 2n$ be a prime, and assume that $\rk J(\Q)=0$
and $J(\Q)=\ClCusp_\Q X$.
Let $c_1,\dots,c_r$ be the Galois orbits of cusps over $\Q$, ordered so that $c_1$ is a
$\Q$-rational cusp, and let $\varphi:\Z^r\to \Cl(X_{\F_p})$ be the homomorphism sending the
$i$-th generator to the class of $\overline{c_i}-\deg(c_i)\,\overline{c_1}$. Then:
\begin{enumerate}
\item[\textup{(a)}] $\Z^r/\ker\varphi\simeq J(\Q)$ canonically, and reduction modulo $p$ is
      injective on $J(\Q)$ with image $\operatorname{im}\varphi$;
\item[\textup{(b)}] if $D\in X^{(5)}(\Q)$ then $\red_p\bigl([D-5c_1]\bigr)+[5\overline{c_1}]$
      lies in the set $S_p$ of classes of effective divisors of degree $5$ on $X_{\F_p}$;
\item[\textup{(c)}] for every $\gamma\in J(\Q)$ the effective divisors of degree $5$ over $\Q$
      in the class of $5c_1+\gamma$ are exactly the elements of the complete linear system
      $|5c_1+\gamma|$.
\end{enumerate}
\end{proposition}

\begin{proof}
Reduction modulo $p$ is injective on $J(\Q)=J(\Q)_\tors$ by \cite[Appendix]{katz}, since $p>2$
is a prime of good reduction of $X$. By hypothesis the map $\Z^r\to J(\Q)$ sending the
$i$-th generator to $[c_i-\deg(c_i)c_1]$ is surjective, and $\varphi$ is its composite with the
injective map $\red_p$. Hence $\ker\varphi$ is the kernel of $\Z^r\to J(\Q)$, giving (a). Part
(b) holds because $\red_p$ commutes with the Abel--Jacobi map and $D$ reduces to an effective
divisor of degree $5$ on $X_{\F_p}$. Part (c) is the definition of the linear system.
\end{proof}

\begin{remark}
\Cref{prop:direct} turns the determination of $X^{(5)}(\Q)$ into a finite computation, and this
is how we use it. We compute $\Cl(X_{\F_p})$ and all closed points of $X_{\F_p}$ of degree at
most $5$, and construct $S_p$. By (a) the group $J(\Q)$ is the finite group $\Z^r/\ker\varphi$, so
we enumerate its elements $\gamma$ and discard those failing the condition in (b). Using several primes $p$ and keeping the $\gamma$ that survive at each of them shortens the list further. For each surviving $\gamma$ we compute the Riemann--Roch space $L(5c_1+\gamma)$ over
$\Q$, whose projectivization is $|5c_1+\gamma|$.
By (b) every element of $X^{(5)}(\Q)$ is found in this way.
\end{remark}

\begin{remark}\label{rem:basecusp}
The choice of base point in \Cref{prop:direct} matters: $c_1$ must be the reduction of a
$\Q$-rational cusp, not merely an $\F_p$-rational cusp of the model. At a prime $p$ at which
some Galois orbit of cusps of degree $>1$ splits into $\F_p$-rational points, an arbitrary
$\F_p$-rational cusp of the model need not be the reduction of a $\Q$-rational one, and then
the subgroup of $\Cl(X_{\F_p})$ generated by the classes
$[\overline{c_i}-\deg(c_i)\,\overline{c}]$ is strictly larger than the image of $\ClCusp_\Q X$.
This is not hypothetical: on $X_1(30)$ modulo $7$ the four Galois orbits (over $\Q$) of cusps of degree $2$
split into $\F_7$-rational cusps, so eight of the sixteen $\F_7$-rational cusps are not reductions of $\Q$-rational ones. Taking $\overline c$ to be one of those eight, we \gitlink{cuspgroup_check.m}{compute} invariant factors $[136,8160]$ instead of $[4,8160]$.
\end{remark}

The hypothesis $J(\Q)=\ClCusp_\Q X$ holds for all the curves we need: for $X_1(n)$ with
$n\leq 55$, $n\neq 37,43,53,54$ this is \cite[Corollary 4.14]{DEHMZ}, and for $X_1(2,2N)$ with
$N\leq16$ it is \cite[Theorem 3.1(3) and Proposition 4.16]{DEHMZ}.

We also explain, once and for all, how to count the effective $\Q$-rational cuspidal divisors of
a given degree. If $X$ has $a_e$ Galois orbits of cusps over $\Q$ of degree $e$, then such a
divisor is a non-negative integer combination of those orbits, so the number of them of degree $d$ is the coefficient
of $t^d$ in
\begin{equation}\label{eq:cuspgf}
\prod_{e\geq1}(1-t^e)^{-a_e}.
\end{equation}
In each proof below the degrees of the Galois orbits of cusps are listed, and the number of
effective $\Q$-rational cuspidal divisors of degree $5$ quoted afterwards is the corresponding
coefficient of $t^5$. These divisors need not lie in distinct classes. They do when $\gon_\Q X>5$, but not on $X_1(2,18)$, where $1980$ divisors fall into $1926$ classes.

\begin{proposition}\label{prop:2_20_2_24}
The groups $\tg{2}{20}$ and $\tg{2}{24}$ are not quintic torsion groups. Hence, by \Cref{lem:4n}, neither are $\torz{40}$ and $\torz{48}$.
\end{proposition}

\begin{proof}
Let $X=X_1(2,20)$. Its Jacobian has $J(\Q)=\ClCusp_\Q X\simeq \torz4\times\torz{60}\times\torz{120}$
of order $28800$, and $\gon_\Q X>5$, since Derickx and Sutherland proved
$\gon_{\F_3}X>5$ \cite[proof of Proposition 5.3]{DS}. There are $16$
Galois orbits of cusps over $\Q$, of degrees $1^8,2^4,4^4$, and carrying out the computation of \Cref{prop:direct} at $p=3$ and at $p=7$ we \gitlink{direct_analysis_X11.m}{obtain} the invariant factors $[4,60,120]$ at both primes. Over $\F_3$ there are $1752$ effective divisors of degree $5$ and over $\F_7$ there are $20920$. Exactly $1384$ of the $28800$ classes of $J(\Q)$ satisfy condition (b) of \Cref{prop:direct} for both $p=3$ and $p=7$. On the other hand, \eqref{eq:cuspgf} gives $1384$ effective $\Q$-rational cuspidal divisors of degree $5$, and they lie in $1384$ distinct
classes, all of which are among the $1384$ candidates. Hence every effective $\Q$-rational
divisor of degree $5$ is linearly equivalent to a cuspidal one, and since $\gon_\Q X>5$, it
\emph{equals} that cuspidal divisor. In particular $X$ has no
non-cuspidal point of degree $5$ (and none of degree $1$, by Mazur).

The argument for $X=X_1(2,24)$ is identical, with $p=5$. Here
$J(\Q)=\ClCusp_\Q X\simeq\torz2\times\torz4\times\torz4\times\torz{120}\times\torz{240}$ of
order $921600$, there are $19$ Galois orbits of cusps of degrees $1^8,2^6,4^5$, and
$\gon_{\F_5}X>5$ by \cite[proof of Proposition 5.3]{DS}, so $\gon_\Q X>5$. Over $\F_5$ there are $14020$ effective divisors of degree $5$, all in distinct classes. Exactly $1720$ classes of $J(\Q)$ satisfy (b), and \eqref{eq:cuspgf} gives $1720$
effective $\Q$-rational cuspidal divisors of degree $5$, lying in $1720$ distinct classes.
\end{proof}

\begin{remark}
The same computation carried out on $X_1(26)$ at $p=7$ gives an independent confirmation of
\Cref{prop:sieve} for $n=26$: there $J(\Q)=\ClCusp_\Q X_1(26)\simeq\torz{133}\times\torz{1995}$,
there are $12$ rational cusps and two Galois orbits of degree $6$, and the $\binom{16}{5}=4368$
classes satisfying (b) are the classes of the $4368$ effective rational cuspidal
divisors of degree $5$.
\end{remark}

We now turn to the three curves that do have quintic points.

\begin{proposition}\label{prop:28}
The non-cuspidal closed points of degree $5$ on $X_1(28)$ are exactly the six diamond
translates $\dm a x_0$, $a\in\Delta(28)$, of van Hoeij's point
\[
x_0=(\alpha,\ \alpha^3-1)\in X_1(28),\qquad \alpha^5-\alpha^4-2\alpha^3-\alpha^2+2\alpha+2=0,
\]
in Sutherland's coordinates \cite{sutherlandtables}.
\end{proposition}

\begin{proof}
We have $\rk J_1(28)(\Q)=0$ and $J_1(28)(\Q)=\ClCusp_\Q X_1(28)\simeq
\torz2\times\torz4\times\torz{12}\times\torz{936}$ of order $89856$
\cite[Theorem 3.1(2), Corollary 4.14 and Table 2]{DEHMZ}, and $\gon_\Q X_1(28)=6$ \cite[Theorem 4 and Table 1]{DvH}, so
every class of degree $5$ contains at most one effective divisor. We use \Cref{prop:direct} with
$p=3$. There are $15$ Galois orbits of cusps over $\Q$, of degrees $1^9,2^3,3,6,6$, and we \gitlink{direct_analysis.m}{compute} that the image of $\ClCusp_\Q X_1(28)$ in $\Cl(X_{\F_3})$ has invariant factors $[2,4,12,936]$, as expected. Over $\F_3$ the curve has $9,3,5,12,54$ closed points of degree
$1,\dots,5$, giving $2238$ effective divisors of degree $5$ in $2238$ distinct classes. Exactly
$1890$ of the $89856$ classes of $J(\Q)$ satisfy condition (b) of \Cref{prop:direct}. Of these,
$1884$ are the classes of the $1884$ effective $\Q$-rational cuspidal divisors of degree $5$
given by \eqref{eq:cuspgf}, which lie in pairwise distinct
classes. The remaining $6$ classes are the classes of the six divisors $\dm a x_0^{(5)}$,
computed by reducing $x_0$ modulo the primes of $\Q(\alpha)$ above $3$ (the prime $3$ is inert
in $\Q(\alpha)$, so $x_0^{(5)}$ reduces to a single closed point of degree $5$). It follows that the effective $\Q$-rational divisors of degree $5$ on $X_1(28)$ are the $1884$ cuspidal ones and the six $\dm a x_0^{(5)}$.
\end{proof}

\begin{remark}
The Hecke sieve at $p=3$ gives the same conclusion but requires an extra step. It leaves five
survivors (one of type $[4]$ and four of type $[5]$, out of $1+2+9$ orbits), of which only one
is the reduction of the known point. The other four are eliminated by checking that the classes
$[D+C-5\overline{c_1}]$ do not lie in the image of $J_1(28)(\Q)$, for every effective cuspidal
complement $C$. This is the
degree $5$ analogue of the last step in the proof of \cite[Proposition 5.13]{DN}.
\end{remark}

\begin{proposition}\label{prop:30}
The non-cuspidal closed points of degree $5$ on $X_1(30)$ are exactly the eight diamond
translates $\dm a x_i$, $a\in\Delta(30)$, $i=1,2$, of van Hoeij's two points
\[
x_1=\bigl(\alpha,\,2\alpha^4+\alpha^3-6\alpha^2+4\alpha+4\bigr),\qquad
x_2=\bigl(\beta,\,\tfrac{1}{53}(3\beta^4+7\beta^3+6\beta^2+11\beta-73)\bigr),
\]
where $\alpha^5+\alpha^4-3\alpha^3+3\alpha+1=0$ and
$\beta^5+\beta^4-7\beta^3+\beta^2+12\beta+3=0$. The two quintic fields $\Q(\alpha)$ and
$\Q(\beta)$ are isomorphic.
\end{proposition}

\begin{proof}
We have $\rk J_1(30)(\Q)=0$ and $J_1(30)(\Q)=\ClCusp_\Q X_1(30)\simeq\torz4\times\torz{8160}$
of order $32640$ \cite[Theorem 3.1(2), Corollary 4.14 and Table 2]{DEHMZ}, and
$\gon_\Q X_1(30)=6$ \cite[Theorem 4 and Table 1]{DvH}. We use \Cref{prop:direct} with $p=7$. There are $16$ Galois
orbits of cusps over $\Q$ of degrees $1^8,2^4,4^4$, and we \gitlink{direct_analysis.m}{compute} that the image of $\ClCusp_\Q$ in $\Cl(X_{\F_7})$ has invariant factors $[4,8160]$. Over $\F_7$ there are $16,20,128,606,3168$
closed points of degree $1,\dots,5$ and $68016$ effective divisors of degree $5$, lying in
$68016$ distinct classes. Exactly $1392$ of the $32640$ classes of $J(\Q)$ satisfy condition
(b). Of these, $1384$ are the classes of the effective $\Q$-rational cuspidal divisors of degree
$5$ counted by \eqref{eq:cuspgf}, and the remaining $8$ are the
classes of $\dm a x_i^{(5)}$ for $i=1,2$ and $a\in\Delta(30)$, obtained by reducing $x_1,x_2$
modulo the primes above $7$ (which is inert in the quintic field). As $\gon_\Q X_1(30)=6>5$,
each class contains at most one effective divisor, which proves the claim.
\end{proof}

\begin{proposition}\label{prop:2_18}
The curve $X_1(2,18)$ has exactly $18$ non-cuspidal closed points of degree $5$. Their
$j$-invariants all have the same minimal polynomial over $\Q$, which is of degree $5$, and
each of the points is the unique effective divisor in its linear equivalence class.
\end{proposition}

\begin{proof}
We have $\rk J_1(2,18)(\Q)=0$ and
$J_1(2,18)(\Q)=\ClCusp_\Q X_1(2,18)\simeq\torz2\times\torz{42}\times\torz{126}$ of order
$10584$ \cite[Theorem 3.1(3), Proposition 4.16 and Table 3]{DEHMZ}. The gonality of
$X_1(2,18)$ is $4$, so we cannot argue as in \Cref{prop:28,prop:30}. Instead we use the fact, proved by Derickx and Sutherland \cite[Proposition 5.4]{DS}, that $\Q(X_1(2,18))$ contains no function of degree \emph{exactly} $5$.

We run \Cref{prop:direct} at $p=5$. There are $15$ Galois orbits of cusps over $\Q$, of degrees
$1^9,2^3,3^3$, and we \gitlink{direct_analysis_X11.m}{compute} that the image of $\ClCusp_\Q$ in $\Cl(X_{\F_5})$ has invariant factors $[2,42,126]$. Over $\F_5$ the curve has $9,3,51,153,612$ closed points of degree $1,\dots,5$,
and its $6273$ effective divisors of degree $5$ fall into $6138$ classes. Exactly $1944$ of the
$10584$ classes of $J(\Q)$ satisfy condition (b) of \Cref{prop:direct}.

By \eqref{eq:cuspgf} there are $1980$ effective $\Q$-rational cuspidal divisors of degree $5$. As the gonality is $4$, they need not lie in distinct classes, and we compute that they fall into $1926$ classes, all among the $1944$ candidates.

We claim that a class containing an effective cuspidal divisor $C$ of degree $5$ contains no
irreducible non-cuspidal divisor. Suppose $D\sim C$ over $\Q$ with $D$ an irreducible
non-cuspidal divisor of degree $5$, and let $f\in\Q(X_1(2,18))^\times$ satisfy
$\ddiv(f)=D-C$. The supports of $D$ and of $C$ are disjoint, since $D$ is non-cuspidal and $C$
is cuspidal, so no cancellation occurs and the polar divisor of $f$ is exactly $C$. Hence $f$ has degree $\deg C=5$. This
contradicts \cite[Proposition 5.4]{DS}.

There remain $1944-1926=18$ classes. For each of them we \gitlink{direct_analysis_X11.m}{compute} the Riemann--Roch space $L(5c_1+\gamma)$ over $\Q$. Each is one-dimensional, so the class contains a unique effective divisor, and we compute that this divisor is an irreducible non-cuspidal closed point of degree $5$. All $18$ have the same
minimal polynomial of the $j$-invariant, of degree $5$. Since every non-cuspidal closed point
of degree $5$ has its class among the $1944$ candidates, and cannot lie in one of the $1926$
cuspidal classes, these $18$ points are all of them.
\end{proof}

\begin{remark}\label{rem:pencils218}
A \gitlink{pencils_2_18.m}{supplementary computation} shows that $27$ of the $1944$ candidate classes contain more than one effective rational divisor of degree $5$, forming a pencil whose
fixed part is a $\Q$-rational cusp, and that each of the remaining $1917$ contains exactly one;
in particular no complete linear system of degree $5$ on $X_1(2,18)$ is base-point free, which
reproves \cite[Proposition 5.4]{DS} for this curve.
\end{remark}

\begin{table}[H]
\centering
\begin{tabular}{@{}llrrr@{}}
\toprule
$X$ & $p$ & candidates & cuspidal & non-cuspidal points of degree $5$\\
\midrule
$X_1(2,20)$ & $3,7$ & $1384$ & $1384$ & $0$\\
$X_1(2,24)$ & $5$   & $1720$ & $1720$ & $0$\\
$X_1(28)$   & $3$   & $1890$ & $1884$ & $6$\\
$X_1(30)$   & $7$   & $1392$ & $1384$ & $8$\\
$X_1(2,18)$ & $5$   & $1944$ & $1926$ & $18$\\
\bottomrule
\end{tabular}
\medskip
\caption{The computations of this section. ``Candidates'' are the classes of $J(\Q)$ satisfying
condition (b) of \Cref{prop:direct} at every prime $p$ listed, and ``cuspidal'' are those among
them that contain an effective $\Q$-rational cuspidal divisor of degree $5$. The last column is
not obtained by subtracting the previous two, as it also uses the uniqueness or linear-system
arguments given in each proof.}
\label{table:direct}
\end{table}

\section{The three sporadic groups and the proof of the main theorem}\label{sec:sporadic}

We now make the three exceptional cases explicit. Let $x,y$ be the coordinates on Sutherland's
plane model of $X_1(n)$ \cite{sutherlandtables}. Following \cite{vanHoeij}, from a
non-cuspidal point $(x_0,y_0)$ of that model one forms
\[
r=\frac{x_0^2y_0-x_0y_0+y_0-1}{x_0^2y_0-x_0},\qquad
s=\frac{x_0y_0-y_0+1}{x_0y_0},\qquad b=rs(r-1),\quad c=s(r-1),
\]
and then
\begin{equation}\label{eq:tate}
E(b,c):\quad y^2+(1-c)xy-by=x^3-bx^2
\end{equation}
has $P=(0,0)$ of order $n$. On the Derickx--Sutherland models of $X_1(2,2n)$
\cite{DSmodels} a point yields $b,c$ such that
\begin{equation}\label{eq:tate2}
E'(b,c):\quad y^2=x^3+cx^2+(1-b-c)bx
\end{equation}
has $P=(0,0)$ of order $2$ and $Q=(b,b)$ of order $2n$.

\begin{proposition}\label{prop:sporadiccurves}
\begin{enumerate}
\item Let $K_{28}:=\Q(\alpha)$ with $\alpha^5-\alpha^4-2\alpha^3-\alpha^2+2\alpha+2=0$ and let
$E_{28}=E(b,c)$ as in \eqref{eq:tate} with
\begin{equation}\label{eq:E28}
b=419\alpha^4-759\alpha^3-222\alpha^2-239\alpha+1032,\qquad
c=17\alpha^4-31\alpha^3-8\alpha^2-11\alpha+42 .
\end{equation}
Then $P=(0,0)$ has order $28$ and $E_{28}(K_{28})_\tors\simeq\torz{28}$.
\item Let $K_{30}:=\Q(\alpha)$ with $\alpha^5+\alpha^4-3\alpha^3+3\alpha+1=0$ and let
$E_{30},E_{30}'$ be $E(b,c)$ as in \eqref{eq:tate} with, respectively,
\begin{align}
b&=-6378\alpha^4-1917\alpha^3+20475\alpha^2-14321\alpha-9118, \notag\\
c&=-114\alpha^4-34\alpha^3+366\alpha^2-257\alpha-162, \label{eq:E30a}\\[4pt]
b&=190\alpha^4-38\alpha^3-382\alpha^2+238\alpha+295,\notag\\
c&=7\alpha^4+13\alpha^3-36\alpha^2+9\alpha+28. \label{eq:E30b}
\end{align}
Then in both cases $P=(0,0)$ has order $30$ and the torsion group is $\torz{30}$. The two
curves are $2$-isogenous over $K_{30}$ and are not $K_{30}$-isomorphic.
\item Let $K_{2,18}:=\Q(w)$ with $w^5-2w^4+2w^3-3w^2-w+4=0$ and let $E_{2,18}=E'(b,c)$ as in
\eqref{eq:tate2} with
\begin{equation}\label{eq:E218}
\begin{aligned}
b&=\tfrac13\bigl(480w^4-1364w^3+2108w^2-3216w+2248\bigr),\\
c&=\tfrac13\bigl(-940w^4+2668w^3-4128w^2+6296w-4389\bigr).
\end{aligned}
\end{equation}
Then $P=(0,0)$ has order $2$, $Q=(b,b)$ has order $18$, and
$E_{2,18}(K_{2,18})_\tors\simeq\tg{2}{18}$.
\end{enumerate}
In all four cases the curve has no complex multiplication and $[\Q(j):\Q]=5$. The invariants of
the three fields and of the four curves are collected in \Cref{table:fields}.
\end{proposition}

\begin{proof}
We \gitlink{verify_sporadic_curves.m}{verify} this directly in Magma \cite{magma}. The
curves of (1) and (2) are obtained from van Hoeij's points of \Cref{prop:28,prop:30} through
the formulas above, and we check that the values of $b$ and $c$ so obtained are the ones
printed in \eqref{eq:E28}, \eqref{eq:E30a} and \eqref{eq:E30b}. Note that in (2) the second
point is defined over a field isomorphic to $\Q(\alpha)$, and we have realized it over
$\Q(\alpha)$ itself. The curve of (3) is obtained from one of the $18$ points of
\Cref{prop:2_18}, whose residue field is $\Q(w)$ in every case.
\end{proof}

\begin{table}[H]
\centering
\begin{tabular}{@{}llllr@{}}
\toprule
$E(K)_\tors$ & $E$ & $K$ (LMFDB label) & $\disc K$ & $\Norm\mathfrak f_E$\\
\midrule
$\torz{28}$  & $E_{28}$   & \lmfdbnf{5.3.17348.1} & $-2^2\cdot 4337$ & $16$\\
$\torz{30}$  & $E_{30}$   & \lmfdbnf{5.3.10407.1} & $-3\cdot 3469$   & $33$\\
$\torz{30}$  & $E_{30}'$  & \lmfdbnf{5.3.10407.1} & $-3\cdot 3469$   & $33$\\
$\tg{2}{18}$ & $E_{2,18}$ & \lmfdbnf{5.3.34779.1} & $-3\cdot 11593$  & $6$\\
\bottomrule
\end{tabular}
\medskip
\caption{The four exceptional pairs $(K,E)$ of \Cref{thm:sporadic}. Here $\mathfrak f_E$ denotes the
conductor of $E$ and $\Norm=\Norm_{K/\Q}$ the absolute norm. Each of the three fields has signature $(3,1)$, class number $1$ and normal
closure with Galois group $S_5$, and each of the four curves is without complex multiplication
and has $[\Q(j):\Q]=5$.}
\label{table:fields}
\end{table}

\begin{proposition}\label{prop:28_30_218}
We use the notation of \Cref{prop:sporadiccurves}. Let $K$ be a quintic field and $E/K$ an elliptic curve.
\begin{enumerate}
\item If $E(K)$ contains a point of order $28$, then $(K,E)\simeq (K_{28},E_{28})$ and
      $E(K)_\tors\simeq\torz{28}$.
\item If $E(K)$ contains a point of order $30$, then $(K,E)$ is isomorphic to
      $(K_{30},E_{30})$ or to $(K_{30},E_{30}')$, and $E(K)_\tors\simeq\torz{30}$.
\item If $E(K)$ contains a subgroup isomorphic to $\tg{2}{18}$, then
      $(K,E)\simeq(K_{2,18},E_{2,18})$ and $E(K)_\tors\simeq\tg{2}{18}$.
\end{enumerate}
In particular no group properly containing $\torz{28}$, $\torz{30}$ or $\tg{2}{18}$ is a
quintic torsion group.
\end{proposition}

\begin{proof}
We prove (1) and (2) together, and then (3), where an extra step is needed. As in
\Cref{sec:groups}, a point of order $28$ on $E(K)$ gives a non-cuspidal closed point $\bar x$
of $Y_1(28)$ of degree $5$, so that $\Q(\bar x)\simeq K$. By \Cref{prop:28} $\bar x$ is one of the six
diamond translates $\dm ax_0$, whose residue field is $K_{28}$ in every case. As $28>4$,
$X_1(28)$ is a fine moduli space, so the universal pair over $\Q(\bar x)$ is $(E_{28},aP)$ for
the corresponding $a$, and after identifying $K$ with $K_{28}$ we get $E\simeq E_{28}$.
Finally $E_{28}(K)_\tors\simeq\torz{28}$ by \Cref{prop:sporadiccurves}. Case (2) is the same,
using \Cref{prop:30} in place of \Cref{prop:28}.

For (3), put $E:=E_{2,18}$ and $K:=K_{2,18}$. We show that the level structures on $E$ account for all $18$ closed points. The extra step is needed because \Cref{prop:2_18} reports only the
number of the quintic points and the minimal polynomial of their $j$-invariants, which leaves
open the possibility of distinct quadratic twists over $K$. On $E$ there are
$18$ points $Q$ of order $18$, and for each of them two points $P$ of order $2$ outside
$\dm Q$. Since $\Aut(E)=\{\pm1\}$ and $(P,Q)$ and $(-P,-Q)=(P,-Q)$ give isomorphic triples, these $36$ pairs give exactly $18$ level structures $\tg{2}{18}\hookrightarrow E(K)$ up to
$K$-isomorphism. Note that they define $18$ \emph{distinct} closed points of $Y_1(2,18)$: if $\sigma\in\Gal(\Qbar/\Q)$ carried one to another, it would fix $j(E)$, hence $K=\Q(j(E))$, hence $E$ and all its $K$-rational points, so it would fix the level structure. By \Cref{prop:2_18} there are exactly $18$ non-cuspidal closed points of degree $5$
on $X_1(2,18)$, so these are all of them.

Now let $E'/K'$ be any elliptic curve over a quintic field with $\tg{2}{18}\subseteq E'(K')$.
As above it gives a non-cuspidal closed point of degree $5$, which must be one of the $18$ just described. Since $X_1(2,18)$ is a fine moduli space, $(K',E')\simeq(K,E)$, and
$E(K)_\tors\simeq\tg{2}{18}$ by \Cref{prop:sporadiccurves}.

In all three cases the last sentence follows, since a group $G$ properly containing one of the three
and realized as $E(K)_\tors$ would in particular contain that group, forcing $E$ to be one of
the four curves and hence $G$ to be its exact torsion group.
\end{proof}

\begin{proposition}\label{prop:36}
$\torz{36}$ is not a quintic torsion group.
\end{proposition}

\begin{proof}
Suppose $E_1$ is an elliptic curve over a quintic field $K$ with a point $P_1\in E_1(K)$ of
order $36$. By \Cref{lem:4n} the triple $\bigl(E',Q',P_1\bmod\dm{18P_1}\bigr)$, where
$E'=E_1/\dm{18P_1}$ and $Q'$ generates $E_1[2]/\dm{18P_1}$, is a $K$-point of $Y_1(2,18)$. Its
closed point has degree $1$ or $5$ over $\Q$, and degree $1$ is impossible because
$\tg{2}{18}\notin\Phi(1)$ by Mazur's theorem. So it has degree $5$, and by
\Cref{prop:28_30_218}(3) we may assume $K=\Q(w)$, $E'=E_{2,18}$ and $Q'\in E_{2,18}[2]$.

Let $\phi:E_1\to E'$ be the quotient isogeny. Its kernel is $\dm{18P_1}$ of order $2$, and
$\ker\hat\phi=\phi(E_1[2])=\dm{Q'}$. Hence $E_1\simeq E'/\dm{Q'}$ is one of the three curves
$2$-isogenous to $E_{2,18}$ over $K$. We \gitlink{verify_sporadic_curves.m}{compute} that each of these three curves has $K$-rational torsion isomorphic to $\torz{18}$. In particular, none of them has a point of order $36$, which is a contradiction.
\end{proof}

\begin{proof}[Proof of \Cref{thm:main,thm:sporadic,thm:points}]
By \Cref{lem:45}, if $G\in\Phi(5)\setminus\Phi^\infty(5)$ then $G$ contains some member of
$\mathcal G$. \Cref{table:summary} lists, for each of the $45$ members of $\mathcal G$, where
it is dealt with. Forty-two of them are shown not to be quintic torsion groups, and for the
remaining three, \Cref{prop:28_30_218} shows that any elliptic curve over a quintic field
containing one of them is one of the four curves of \Cref{prop:sporadiccurves}, whose torsion
groups are $\torz{28}$, $\torz{30}$, $\torz{30}$ and $\tg{2}{18}$. Hence
$G\in\{\torz{28},\torz{30},\tg{2}{18}\}$, and conversely all three occur. Together with
\eqref{eq:phiinf5} this proves \Cref{thm:main}, and \Cref{prop:sporadiccurves,prop:28_30_218}
prove \Cref{thm:sporadic}.

We now prove \Cref{thm:points}. The counts of closed points are \Cref{prop:28,prop:30,prop:2_18}. On
$X_1(28)$ the six points are the six diamond translates of $x_0$, and
$\#\Delta(28)=\varphi(28)/2=6$. A group of order $6$ acting transitively on a set of $6$ elements acts freely, so the six points form a single free orbit. They correspond to the twelve
points of order $28$ of $E_{28}(K)_\tors\simeq\torz{28}$, taken up to sign. Similarly
$\#\Delta(30)=\varphi(30)/2=4$ and the eight points of $X_1(30)$ split into two free orbits of
length $4$, corresponding to the eight points of order $30$ on each of $E_{30}$, $E_{30}'$
taken up to sign. On $X_1(2,18)$ the $18$ points of \Cref{prop:2_18} are the $18$ level
structures $\tg{2}{18}\hookrightarrow E_{2,18}(K)$ counted in the proof of
\Cref{prop:28_30_218}.
\end{proof}

\subsection{Points of degree five on \texorpdfstring{$X_1(n)$ and $X_1(2,2n)$}{X\_1(n) and X\_1(2,2n)}}
We can now prove \Cref{cor:degree5points} and describe the sporadic and isolated points
involved.
\begin{proof}[Proof of \Cref{cor:degree5points}]
For $n\leq4$ the curve $X_1(n)$ has genus $0$ and a rational point, so (1) is clear. For $n\geq5$ it is a fine moduli space, and a non-cuspidal closed point of degree $5$ gives an elliptic curve over its quintic residue field with a point of order $n$, so $n$ divides the exponent of some group in $\Phi(5)$. By \Cref{thm:main} the exponents that occur are
$1,\dots,22,24,25,28,30$, a set closed under divisors. That $X_1(n)$ has infinitely many points
of degree $5$ for $n\leq25$, $n\neq23$, is \cite[Theorem 3]{DvH}, and the counts for $n=28,30$
are \Cref{thm:points}. The same argument gives (2), using \cite[Proposition 5.3]{DS} for
$n\leq8$, while for $n\geq10$ a non-cuspidal quintic point would give an $E(K)_\tors\in\Phi(5)$
containing $\tg{2}{2n}$, whereas every non-cyclic group in $\Phi(5)$ has exponent at most
$18<2n$. Part (3) follows because a curve with infinitely many points of degree $5$ has none
that is sporadic, while $X_1(28)$ and $X_1(30)$ have $\gon_\Q=6$ and rank $0$ Jacobians, hence
only finitely many points of degree $\leq5$ \cite[Propositions 2.2 and 2.3]{DS}, and
$X_1(2,18)$ has infinitely many points of degree $4$.
\end{proof}

\begin{remark}\label{rem:cuspidalsporadic}
The word \emph{non-cuspidal} cannot be dropped from \Cref{cor:degree5points}(3). On $X_1(44)$
the part of the cuspidal subscheme indexed by $d=2$ consists of $\varphi(22)\varphi(2)/2=5$
geometric points \cite[Lemma 2.8]{DEHMZ}, on which $\Gal(\Qbar/\Q)$ acts through
$(\torz{22})^\times$ with stabilizer $\{\pm1\}$, so they form a single closed point of degree $5$, with residue field $\Q(\zeta_{11})^+$. Since $\rk J_1(44)(\Q)=0$
\cite[Theorem 3.1(2)]{DEHMZ} and $\gon_\Q X_1(44)>8$ \cite[proof of Proposition 6]{DvH}, the
curve has only finitely many points of degree $\leq5$ \cite[Proposition 2.3]{DS}, so this
cuspidal point is sporadic.
\end{remark}

The quintic points of \Cref{thm:points}(3) are not sporadic, since $\gon_\Q X_1(2,18)=4$, so $X_1(2,18)$ has infinitely many quartic points. They are nevertheless \emph{isolated} in the sense of Bourdon, Ejder, Liu, Odumodu and Viray \cite[Definition 4.1]{BELOV}. The curve has only finitely many points of degree $5$, so by \cite[Theorem 4.2\,(1)]{BELOV} all of them are isolated.

\begin{remark}
The curve $E_{2,18}$ has conductor of norm $6$: it has multiplicative reduction exactly at the
prime of norm $2$ and the prime of norm $3$ of $K$. This is very small: for comparison, the curves with $28$- and $30$-torsion have conductors of norm $16$ and $33$.
\end{remark}

\section{The computations}\label{sec:computations}

All of our computations were performed in Magma V2.29-4 \cite{magma} on the \emph{Mordell} server at the Department of Mathematics, University of Zagreb, with an AMD EPYC 9175F CPU and $384$ GB of RAM, using the package
\texttt{mdmagma} \cite{mdmagma} for the moduli-theoretic operations (the routines
\texttt{ModuliPoint}, \texttt{LevelStructure}, \texttt{HeckeOperator},
\texttt{DiamondOperator}, \texttt{CuspOrbitsQ}), Sutherland's optimized models of $X_1(n)$
\cite{sutherland,sutherlandtables} and the Derickx--Sutherland models of $X_1(2,2n)$ \cite{DSmodels}.
The total cost of the computations reported here was about $203$ hours of CPU time, of which
$197$ hours were spent on the four curves $X_1(45)$, $X_1(57)$, $X_1(63)$ and $X_1(65)$ and
$4.4$ hours on $X_1(42)$. 

\subsection{Fast Hecke operators and diamond orbits}\label{sec:fast}
On $X_1(57)$, $X_1(63)$ and $X_1(65)$ (of genus $85$, $97$ and $121$, with plane models of
bidegrees $(39,37)$, $(48,45)$ and $(54,53)$) the computations were dominated by two
subroutines of \cite{mdmagma}: the Hecke operator, which builds the splitting field of the whole
$q$-division polynomial, and the reduction of a set of places to diamond orbit representatives,
which performs one function-field computation per diamond image. We replaced both.

\begin{lemma}\label{lem:fasthecke}
Let $q\nmid pn$ be a prime and let $z$ be a non-cuspidal closed point of $X_1(n)_{\F_p}$ of
degree $k$, corresponding to a pair $(E,P)$ over $\kappa:=\F_{p^k}$. Let $\mathcal O$ be the set
of orbits of $\Gal(\overline{\kappa}/\kappa)$ on the set of the $q+1$ subgroups of order $q$ of
$E[q]$. For $O\in\mathcal O$ pick $C\in O$, let $L$ be the field of definition of $C$, and let
$y_O$ be the closed point of $X_1(n)_{\F_p}$ containing $(E/C,P\bmod C)$. Then
\[
T_q(z)=\sum_{O\in\mathcal O}\frac{[L:\kappa]\cdot k}{\deg y_O}\;y_O ,
\]
and $[L:\kappa]=\#O$. Furthermore, $L$ is the field generated over $\kappa$ by the coefficients of
the kernel polynomial $g_C=\prod_{Q\in (C\setminus\{0\})/\pm}\bigl(T-x(Q)\bigr)$, where $x(Q)$
denotes the $x$-coordinate of $Q$ in a fixed Weierstrass model of $E$.
\end{lemma}

\begin{proof}
By definition $T_q(z)=\sum_{\sigma\in\Gal(\kappa/\F_p)}\sum_{C}\sigma\bigl(E/C,P\bmod C\bigr)$,
where $C$ ranges over all $q+1$ subgroups of order $q$ of $E[q]$. Grouping the $k(q+1)$ geometric points by the closed point that contains them, and using that Galois acts
transitively on the geometric points of a closed point, gives
$T_q(z)=\sum_y (m_y k/\deg y)\,y$ with $m_y=\#\{C:(E/C,P\bmod C)\in y\}$. This is the formula
implemented in \cite{mdmagma}. Now $E$ and $P$ are $\kappa$-rational, so for
$C'=F^i(C)$ with $F$ the $\kappa$-Frobenius we have
$(E/C',P\bmod C')=F^i\bigl(E/C,P\bmod C\bigr)$, and these two geometric points lie in the same
closed point of $X$. Hence $m_y=\sum_{O:\,y_O=y}\#O$, and it suffices to run over one $C$ per
orbit together with $\#O$. By the orbit--stabilizer theorem $\#O=[L:\kappa]$ where $L$ is the
fixed field of the stabilizer of $C$. Finally a subgroup of order $q$ is determined by, and
determines, the set of $x$-coordinates of its non-zero points, so the stabilizer of $C$ equals
the stabilizer of $g_C$, i.e.\ $L=\kappa(\text{coefficients of }g_C)$.
\end{proof}

In the implementation, which assumes $q$ odd, a generator of $C$ is obtained from a root $x_1$ of an irreducible factor of the $q$-division polynomial of $E$ over $\kappa$. The remaining $x$-coordinates $x(jQ)$, $1\leq j\leq (q-1)/2$, are rational functions of $x_1$ over $\kappa$ and
hence lie in $\kappa(x_1)$. The isogeny $E\to E/C$ is computed by V\'elu's formulas over $L$, and the closed
point $y_O$ by \texttt{mdmagma}'s \texttt{ModuliPoint}. The factors of the division polynomial
belonging to the orbit $O$ are the irreducible factors of the norm $N_{L/\kappa}(g_C)=\prod_{C'\in O}g_{C'}$, which lets us process each orbit exactly once. The \gitlink{fast_hecke.m}{code} asserts both $\sum_{O}\#O=q+1$ and, for each orbit, that the degrees add up correctly. Every sieve prime used in this paper is odd, and the code rejects $q=2$. Note that the lemma itself is valid for $q=2$ as well, where $g_C$ is linear.

\begin{lemma}\label{lem:fastorbits}
Let $x,x'$ be non-cuspidal closed points of $X_1(n)_{\F_p}$ of the same degree $k$, represented
by $(E_1,P_1)$ over $\kappa_1$ and $(E_2,P_2)$ over $\kappa_2$. Then $x'=\dm ax$ for some
$a\in(\torz n)^\times$ if and only if there exist a field isomorphism $\kappa_2\to\kappa_1$, a
power $\sigma$ of the Frobenius of $\kappa_1$ and an isomorphism
$\iota:E_2^{\sigma}\to E_1$ defined over $\kappa_1$ such that $\iota(P_2^{\sigma})$ generates
$\dm{P_1}$.
\end{lemma}

\begin{proof}
The diamond operators commute with the Galois action, so $x'=\dm ax$ if and only if the
geometric point $(E_1,aP_1)$ is $\Gal(\overline{\F}_p/\F_p)$-conjugate to $(E_2,P_2)$, i.e. if and only if there are $\sigma$ and an isomorphism of pairs
$\iota:(E_2,P_2)^{\sigma}\to (E_1,aP_1)$ over $\overline{\F}_p$. Such an $\iota$ is
automatically defined over $\kappa_1$, since otherwise $\iota^{F}\iota^{-1}$, where $F$ is the
$\kappa_1$-Frobenius $x\mapsto x^{p^k}$, would be a
non-trivial automorphism of $E_1$ fixing the point $aP_1$, which is $\kappa_1$-rational and of
order $n\geq5$. That is impossible, because for a non-trivial automorphism $\zeta$ of an
elliptic curve $\ker(\zeta-1)$ has at most four geometric points, the bound being attained by
$\zeta=[-1]$, for which $\ker(\zeta-1)=E[2]$.
\end{proof}

The implementation loops over the $k$ Frobenius powers $\sigma$, over the isomorphisms
$E_2^\sigma\to E_1$ (one returned by \texttt{IsIsomorphic}, composed with all
$\kappa_1$-rational automorphisms of $E_1$) and over $a\in(\torz n)^\times$. Points are first grouped by the minimal polynomial of $j$ over $\F_p$ and by degree, which cannot split an orbit because the diamond operators preserve both. On $X_1(45)_{\F_7}$ the degree $5$ place for which
\texttt{mdmagma} needed $7114$ seconds to compute $T_{11}$ is handled in $4.3$ seconds by
\Cref{lem:fasthecke}, and on $X_1(57)_{\F_5}$ the orbit reduction of the $756$ places of degree
$5$ takes $11$ seconds instead of $16$--$21$ hours. We validated both routines against \texttt{mdmagma}: the divisors $T_q(x)$ \gitlink{test_fast_hecke.m}{agree} on $12$ places for each of several
$(n,p,\deg x,q)$ on $X_1(26)_{\F_7}$, $X_1(28)_{\F_5}$ and $X_1(30)_{\F_7}$,
and on $15$ curve/degree pairs the orbits of the returned
representatives \gitlink{test_fast_orbits.m}{are pairwise disjoint}, cover the input set, and are equal in number to
\texttt{mdmagma}'s.

\bibliographystyle{amsalpha}
\bibliography{quintic}

\end{document}